\documentclass[a4paper,11pt]{amsart}
\usepackage[top=2.2cm,left=2.8cm,right=2.8cm,bottom=1.8cm]{geometry}

\usepackage[foot]{amsaddr}

\usepackage{hyperref}
\hypersetup{
    colorlinks,
    citecolor=green,
    filecolor=black,
    linkcolor=blue,
    urlcolor=blue
}
\hypersetup{linktocpage}

\usepackage{mathtools}
\usepackage{cite}
\usepackage{graphicx}
\usepackage{subcaption}
\usepackage{cleveref}
\usepackage{amsbsy}
\usepackage{latexsym}
\usepackage{amsfonts}
\usepackage{amssymb}
\usepackage{upgreek}
\usepackage[usenames]{color}
\usepackage{nicematrix}
\usepackage{amsmath,amsthm}
\usepackage{enumerate}

\usepackage{stmaryrd}
\usepackage{dsfont}

\DeclareMathOperator{\vol}{vol}
\DeclareMathOperator{\cost}{cost}

\DeclareMathOperator{\supp}{supp}
\DeclareMathOperator{\diam}{diam}
\DeclareMathOperator{\meas}{meas}
\DeclareMathOperator{\Span}{span}
\DeclareMathOperator*{\einf}{ess\,inf}

\DeclareMathOperator{\level}{level}

\providecommand{\corol}[1]{Cor.~#1}
\providecommand{\propo}[1]{Prop.~#1}
\providecommand{\sect}[1]{Sec.~#1}
\providecommand{\lem}[1]{Lem.~#1}
\providecommand{\theo}[1]{Thm.~#1}
\providecommand{\alg}[1]{Alg.~#1}

\providecommand{\abs}[1]{\lvert#1\rvert}
\providecommand{\bigabs}[1]{\bigl\lvert#1\bigr\rvert}

\providecommand{\biggabs}[1]{\biggl\lvert#1\biggr\rvert}

\providecommand{\norm}[1]{\lVert#1\rVert}
\providecommand{\bignorm}[1]{\bigl\lVert#1\bigr\rVert}
\providecommand{\Bignorm}[1]{\Bigl\lVert#1\Bigr\rVert}
\providecommand{\biggnorm}[1]{\biggl\lVert#1\biggr\rVert}

\providecommand{\ceil}[1]{\lceil#1\rceil}

\newtheorem{theorem}{Theorem}
\newtheorem{lemma}[theorem]{Lemma}
\newtheorem{prop}[theorem]{Proposition}

\theoremstyle{definition}
\newtheorem{definition}[theorem]{Definition}

\theoremstyle{remark}
\newtheorem{remark}[theorem]{Remark}

\numberwithin{equation}{section}
\numberwithin{theorem}{section}

\theoremstyle{plain}
\newtheorem{assumption}{Assumption}

\newcommand{\Ccomp}{C_{\mathrm{compr}}}
\newcommand{\Ccost}{C_{\mathrm{cost}}}

\newcommand{\cstable}{{c_\mathrm{stable}}}
\newcommand{\thetaelementnum}{\tau}

\newcommand{\cA}{{\mathcal{A}}}

\newcommand{\cM}{{\mathcal{M}}}
\newcommand{\cN}{{\mathcal{N}}}
\newcommand{\cF}{{\mathcal{F}}}

\newcommand{\cV}{\mathcal{V}}
\newcommand{\cT}{\mathcal{T}}

\newcommand{\ba}{\mathbf{a}}

\newcommand{\bz}{\mathbf{z}}

\newcommand{\br}{\mathbf{r}}

\newcommand{\sdd}{\,\mathrm{d}}

\newcommand{\PP}{\mathbb{P}}

\newcommand{\Chi}{\raise .3ex
\hbox{\large $\chi$}} 

\newcommand{\T}{\mathbb{T}}
\newcommand{\R}{\mathbb{R}}

\newcommand{\N}{\mathbb{N}}
\newcommand{\Z}{\mathbb{Z}}
\newcommand{\C}{\mathbb{C}}

\newcommand{\msx}[1]{\text{\textsc{Mesh}}(#1)}

\newcommand{\gknuprimepol}{\hat g_{\nu, \nu'}}
\newcommand{\sublev}{Q}
\newcommand{\jnum}{j}
\newcommand{\Jnum}{\mathcal J}
\newcommand{\opApproxnum}{n}

\newcommand{\initmesh}{\hat\cT_0}

\usepackage[section]{algorithm}
\usepackage{algorithmicx}
\usepackage{algpseudocode}

\usepackage{todonotes}

\title[
  Adaptive SGFEM: Optimality and non-affine coefficients
  ]{
  Adaptive stochastic Galerkin finite element methods: Optimality and non-affine coefficients
}

\author{Markus Bachmayr$^1$} 
\email{bachmayr@igpm.rwth-aachen.de}
\author{Henrik Eisenmann$^1$}
\email{eisenmann@igpm.rwth-aachen.de}
\author{Igor Voulis$^2$}
\email{i.voulis@math.uni-goettingen.de}
\address{$^1$ Institut f\"ur Geometrie und Praktische Mathematik, RWTH Aachen University, Templergraben 55, 52062 Aachen, Germany}
\address{$^2$ Institute for Numerical and Applied Mathematics, University of G\"ottingen, Lotzestr.\ 16-18, 37083 G\"ottingen, Germany}
\date{\today}

\thanks{M.B.\ acknowledges funding by Deutsche Forschungsgemeinschaft (DFG, German Research Foundation) -- project numbers 442047500, 501389786. The work of H.E.~was funded by Deutsche Forschungs\-gemeinschaft – project number 501389786. }

\date{\today}

\begin{document}

\maketitle

\begin{abstract}
Near-optimal computational complexity of an adaptive stochastic Galerkin method with independently refined spatial meshes for elliptic partial differential equations is shown. 
The method takes advantage of multilevel structure in expansions of random diffusion coefficients and combines operator compression in the stochastic variables with error estimation using finite element frames in space. A new operator compression strategy is introduced for nonlinear coefficient expansions, such as diffusion coefficients with log-affine structure.

\smallskip


\end{abstract}

\section{Introduction}

This paper deals with adaptive numerical methods for parameter-dependent or random elliptic partial differential equations (PDEs) based on a joint variational formulation in spatial and parametric variables, where the latter can be interpreted as random variables on which the data and thus the solutions of the considered PDE depend. We treat here the case of infinitely many parameters, which  is common when random fields are represented in series expansions.

We analyze the optimality properties of the adaptive stochastic Galerkin finite element method that was proposed and shown to converge in \cite{bachmayr2024convergent} for problems with parametric diffusion coefficients. This method generates sparse expansions in terms of product Legendre polynomials in the stochastic variables, where each Legendre coefficient is a spatial finite element function on an individually refined mesh. We extend the method from \cite{bachmayr2024convergent}, which is designed for affine coefficient parametrizations, to more general smooth nonlinear parametrizations by a new operator compression scheme. For both cases, we show optimality with respect to the total number of degrees of freedom of the generated approximations and, provided that an expansion with multilevel structure is chosen for the parametric coefficient, near-optimal computational costs of the method.

\subsection{Uniformly elliptic parametric PDEs}

On a polygonal domain $D \subset \R^d$, where typically $d \in \{1,2,3\}$, we consider the elliptic model problem 
\begin{equation}\label{eq:diffusionequation}
   - \nabla \cdot ( a \nabla u ) = f \quad \text{on $D$,} \qquad u = 0 \quad \text{on $\partial D$,}
\end{equation}
in weak formulation with $f \in L_2(D)$.  Assuming $\cM_0$ to be a countable index set with $0 \in \cM_0$ and taking $\cM = \cM_0 \setminus \{ 0 \}$, the parameter-dependent coefficient $a$ is specified by a function series in terms of given $\theta_\mu \in L_\infty(D)$ for $\mu \in \cM_0$, parametrized by vectors of scalar coefficients $y = (y_\mu)_{\mu \in \cM} \in Y$ with the parameter domain $Y = [-1,1]^\cM$. 

A standard test case is the affine parametrization
\begin{equation} \label{eq:affinecoeff}
a(y) = \theta_0 + \sum_{\mu\in \cM} y_\mu \theta_\mu, \quad y = (y_\mu)_{\mu\in\cM} \in Y.
\end{equation}
Nonlinear parametrizations are also of frequent interest; one example is the log-affine pa\-ra\-me\-tri\-za\-tion
\begin{equation} \label{eq:loglinearcoeff}
  a(y) = \exp\biggl(\theta_0 + \sum_{\mu\in \cM} y_\mu \theta_\mu\biggr)
  =\exp(\theta_0)\prod_{\mu\in \cM} \exp(y_\mu \theta_\mu)
  , \quad y  \in Y,
\end{equation}
but in the method developed in this work, we also admit more general smooth nonlinearities.

For any such choice of parametrization, we write
\[
  c_B = \inf_{y \in Y} \einf_{D} a(y) .
\]
In the case of~\eqref{eq:affinecoeff}, well-posedness of \eqref{eq:diffusionequation} uniformly for all $y \in Y$ is ensured by the \emph{uniform ellipticity condition} \cite{CDS:11}
\[ 
    c_B = \einf_D \biggl\{  \theta_0 - \sum_{\mu \in \cM} \abs{\theta_\mu} \biggr\} > 0 ,
\]
which entails the requirement $\einf_D \theta_0 >0 $. In the log-affine case \eqref{eq:loglinearcoeff}, it suffices to require
\[ 
    \log(c_B) = \einf_D \biggl\{  \theta_0 - \sum_{\mu \in \cM} \abs{\theta_\mu} \biggr\}  >-\infty \,.
\]
For $C_B = \sup_{y \in Y} \norm{ a(y) }_{L_\infty(D)}$, we obtain
\begin{equation}\label{eq: CBlinear}
  C_B  \leq  \begin{cases}  \displaystyle 2 \norm{\theta_0}_{L_\infty(D)} - c_B , & \text{assuming~\eqref{eq:affinecoeff},} \\[6pt]
       \displaystyle c_B^{-1} \exp\bigl( 2 \norm{\theta_0}_{L_\infty(D)} \bigr) , & \text{assuming~\eqref{eq:loglinearcoeff}.}\end{cases}
\end{equation}

\subsection{Stochastic Galerkin methods}

We consider the solution map $y\mapsto u(y)$ as an element 
\[    \cV = L_2(Y,V,\sigma) \simeq V\otimes L_2(Y,\sigma),   \]
with $V = H_0^1(\Omega)$ and where for simplicity, we take $\sigma$ to be the uniform measure on $Y$. This corresponds to assuming the scalar random coefficients $y_\mu$ to be independently uniformly distributed; however, other product distributions on $Y$ could be treated in a completely analogous manner.

The variational formulation of~\eqref{eq:diffusionequation} in both spatial and stochastic variables on the space $\cV$ reads:
find $u \in \cV$ such that for all $v \in \cV$,
\begin{equation}\label{eq:weakdiffusioneq}
\langle B u ,v\rangle_{\cV', \cV}=
  \int_Y\int_D a(y)\nabla u(y)\cdot \nabla v(y) \sdd x\sdd\sigma(y) 
=
\int_Y \int_D f\, v(y) \sdd x  \sdd\sigma(y).
\end{equation}
Under uniform ellipticity, we have the equivalence of norms 
\begin{equation}\label{eq:energienormequiv}
c_B\norm{v}_\cV^2 \leq \langle Bv, v\rangle = \norm{v}_B^2 \leq C_B\norm{v}_\cV^2.
\end{equation}
As a consequence, there is a unique solution of~\eqref{eq:weakdiffusioneq} that continuously depends on $f$, and this solution agrees with the solution of \eqref{eq:diffusionequation} pointwise in $y \in Y$.

We consider approximation spaces of the form
\begin{equation}\label{eq:ansatzspaces}
  \cV_N = \Bigl\{\sum_{\nu\in F} v_\nu \otimes L_\nu\colon v_\nu\in V_\nu, \nu\in F \Bigr\}\subset \mathcal V
\end{equation}
with a finite index set $F$, product Legendre polynomials $L_\nu$, $\nu\in F$, and finite-dimensional subspaces $V_\nu\subset V$ for each $\nu\in F$. It is important to stress that the spaces $V_\nu$ are allowed to differ for each $\nu$. This is especially beneficial if the functions $\theta_\mu$ have local supports. Then also the functions $u_\nu$ in an expansion of $u$ have local features, that can only be effectively captured by individual spatial discretizations.

The method analysed in this work is based on the stochastic Galerkin variational formulation for $u_N\in \cV_N$,
\[
\langle B u_N , v\rangle_{\cV',\cV} = \int_Y \int_D f \, v(y)\sdd x \sdd\sigma(y)\quad\text{for all $v\in \cV_N $}.
\]
By C\'ea's lemma, the best approximation $u_N$ to the solution $u$ in the energy norm is also a quasi-best approximation in the $\cV$-norm.

\subsection{A convergent adaptive method}

In~\cite{bachmayr2024convergent}, we introduced an adaptive Galerkin method for solving~\eqref{eq:diffusionequation} in the formulation \eqref{eq:weakdiffusioneq} using finite elements as spatial approximation spaces in~\eqref{eq:ansatzspaces}. This method ensures energy norm error reduction by a uniform factor in each step, which is also known as \emph{saturation property}, and thus convergence in $\cV$-norm. 

As discussed in detail in \cite[Sec.~1.3]{bachmayr2024convergent}, the best possible convergence rates of approximations with respect to the total number of degrees of freedom can be achieved when coefficient parametrization functions $\theta_\mu$, $\mu \in \mathcal M$, with multilevel structure are combined with independently refined adaptive meshes for the Legendre coefficients. This complicates adaptive error estimation and mesh refinement. 
The task amounts to extracting spatially localized error information from the coupled elliptic systems of equations for the individual Legendre coefficients that result from stochastic Galerkin discretization. Here standard techniques for error estimation are infeasible due to the impracticality of ensuring Galerkin orthogonality. Moreover, due to the interaction of independently refined meshes for the different Legendre coefficients, the residual involves facet functionals on potentially strongly varying levels of refinement. As a consequence, a standard cycle of marking elements followed by a single mesh refinement cannot guarantee uniform error reduction~\cite{CDN:12}.

In \cite{bachmayr2024convergent}, we overcome these issues by error estimation based on coefficients with respect to \emph{finite element frames} \cite{HS:16,HSS:08}, that is, families of hat functions on all levels of a refinement hierarchy. This approach does not require Galerkin orthogonality and provides a natural handling of edge terms. Uniform error reduction is ensured by allowing for several steps of mesh refinement in each iteration of the method by on a tree coarsening procedure \cite{B:18,BV}.

The relevant subset of interactions between different Legendre coefficients under the action of the operator $B$ for which this spatial error estimation and refinement is carried out is determined by a semi-discrete operator compression in the parametric variables. That this can be realized efficiently depends strongly on the multilevel structure of the coefficient parametrization. This strategy, adapted from our previous work \cite{BV} (where it was used with wavelet-based spatial discretizations), in \cite{bachmayr2024convergent} was applied only to affine parametrizations \eqref{eq:affinecoeff}. In the present work, we make use of the flexibility that this approach offers to also apply it to nonlinear coefficient parametrizations.

\subsection{Relation to previous work and novelty}

The results on adaptive approximability of solutions with multilevel coefficient representations that are relevant for the present work have been obtained in \cite{BCDS:17} for the affine case \eqref{eq:affinecoeff}, as summarized in \cite{BV,bachmayr2024convergent}. Polynomial approximations for problems with nonlinear coefficient parametrizations have been analyzed under similar assumptions in \cite{MR4799248}, but in a semi-discrete setting without spatial discretization.

Previous work on the complexity analysis of adaptive stochastic Galerkin methods focused on wavelet-based spatial discretizations. While the methods analyzed in \cite{G:14,BCD:17} are not of optimal complexity, optimal convergence with near-optimal computational costs was shown for the method in \cite{BV}, which can be regarded as a precursor to the method using finite elements in \cite{bachmayr2024convergent} that we analyze in the present work. Here, we thus close the gap to the wavelet-based method in \cite{BV} in that we obtain analogous optimality properties, but with the additional flexibility of finite element discretizations, for example in the much simpler treatment of nontrivial domain geometries.

In the case of stochastic Galerkin FEM, we are not aware of any previous complete analysis of the complete computational complexity of such a method. For the method in \cite{Gittelson:13}, which also used a type of operator compression in the parametric variables, a semidiscrete analysis is performed that does not fully account for spatial discretizations. In \cite{EGSZ:15}, a method using the same mesh for all Legendre coefficients is shown to yield quasi-optimal mesh cardinalities, but without control on the overall costs. Methods using independent meshes have been considered in  \cite{EGSZ:14,CPB:19,BPR21}. The convergence analysis in \cite{CPB:19,BPR21}, however, starts from assuming the saturation property, and no analysis of the computational costs is available.

While nonlinear coefficient parametrizations can in certain cases such as \eqref{eq:loglinearcoeff} be reduced to advection-diffusion problems with affine coefficients \cite{UEE:12}, a direct treatment as done here has the advantage of preserving the structure of the elliptic problem in the construction of adaptive methods. For a single spatial mesh, convergence with uniform error reduction is shown for uniformly elliptic and bounded coefficients in \cite{EH:23}, but without complexity estimates. To the best of our knowledge, the scheme analyzed here is the first adaptive method to guarantee convergence and optimal approximations for such nonlinear parametrizations with independent meshes for the Legendre coefficients.

Our assumptions lend themselves to an extension of the method considered here to problems with domain uncertainty, where a family of parameter-dependent domains is transformed to a reference domain \cite{MR3563282,MR4799248}.
This results in matrix-valued parametric diffusion tensors, which that are rational functions with respect to each scalar parameter $y_\mu$.
Although the results of the present work are formulated for scalar-valued diffusion coefficients, we expect that they can be extended to this case.

\subsection{Outline} In \Cref{sec:adaptgalerkin}, we review the adaptive stochastic Galerkin finite element method introduced in \cite{bachmayr2024convergent} and formulate our main results concerning nonlinear coefficient parametrizations and the optimality properties of the method. The construction of new operator compression methods for the nonlinear parametrizations is carried out in \Cref{sec:OperatorCompression}. The computational costs of these new approximations are investigated in \Cref{sec:Complexity}.
We then prove our main results in \Cref{sec:Optimality} and summarize our findings in \Cref{sec:Conclusion}.



\section{Adaptive Galerkin Method}\label{sec:adaptgalerkin}

\subsection{Notation}\label{sec:Notation}
We consider the problem~\eqref{eq:diffusionequation} in variational form on $\cV = L_2(Y, V, \sigma)$ as in~\eqref{eq:weakdiffusioneq}.
For the discretization of $V$, we consider conforming simplicial finite element meshes that are generated from an initial conforming mesh~$\initmesh$ for $D\subset{\R^d}$ by bisection, as analyzed for general $d$ in~\cite{Stevenson:08}. 
We write $\cT\geq \tilde \cT$ if $\cT$ (which need not be conforming) can be generated from $\tilde\cT$ by bisection.
On such conforming meshes, we consider the standard Lagrange finite element spaces 
\begin{equation}\label{eq:defV}
  V(\cT) = \PP_1(\cT) \cap V,
\end{equation}
where $\PP_1(\cT)$ denotes the functions on $D$ that are piecewise affine on each element of $\cT$. 
For $L_2(Y,\sigma)$ we use the orthonormal basis of product Legendre polynomials
\[
  L_\nu(y) = \prod_{\mu\in\cM} L_{\nu_\mu}(y_\mu), \quad \nu\in\cF = \bigl\{ \nu \in \N_0^\cM \colon \#\supp \nu < \infty \bigr\}.
  \]
  
We now assume a family of simplicial meshes $\mathbb{T} = (\cT_\nu)_{\nu \in F}$ with finite $F\subset \cF$ and conforming $\cT_\nu \geq \hat\cT_0$ for each $\nu \in F$.
Then we consider the stochastic Galerkin discretization subspaces $\cV(\mathbb{T})$ given by
\begin{equation}\label{eq:discrspace}
  \cV(\mathbb{T}) = \biggl\{ \sum_{\nu \in F} v_\nu L_\nu \colon  v_\nu \in V(\cT_\nu),\ \nu \in F  \biggr\} \subset \cV.
\end{equation}
The total number of degrees of freedom for representing each element of $\cV(\mathbb{T})$ is then $\dim \cV(\mathbb{T})=\sum_{\nu \in F} \dim V(\cT_\nu)$. We use the abbreviation
\begin{equation}\label{eq:Ndef}
  N(\mathbb{T}) = \sum_{\nu \in F} \# \cT_\nu ,
\end{equation}
so that $N(\mathbb{T}) \eqsim \dim \cV(\mathbb{T})$. 

With the initial mesh $\hat{\mathcal T}_0$, for $j\geq 1$, let $\hat{\mathcal T}_j$ be a refinement of $\hat{\mathcal T}_{j-1}$ such that each element is bisected $d$ times, meaning that each edge is bisected once~\cite{Stevenson:08}. 
We use the nodal basis functions $\psi_{j,k}\in V(\hat\cT_j) = \PP_1(\hat\cT_j) \cap V$ with the property
\begin{equation}\label{eq:NodalBasisFunctions}
  \norm{\psi_{j,k}}_V = 1
  \quad{\text{and}}\quad
  \psi_{j,k}(x_{k'}) = 0 \quad\text{if and only if $k = k'$,}
\end{equation}
where $x_{k'}$ are the vertices of $\hat\cT_j$.  We denote the set of corresponding indices by
\begin{equation}\label{eq:enumerationset}
  \cN_j = \{k\colon \text{$x_k$ is a node in $\hat\cT_j$}\}  ,
\end{equation}
so that $\cN_j \supset \cN_{j-1}\supset\dots \supset \cN_{0}$. Finally, let 
\begin{equation}\label{eq:Thetaset}
  \Theta = \{\lambda = (j,k)\colon j\in \N_0, k \in \cN_j\}
\end{equation}
be the index set for the functions $\psi_\lambda$.

We define $\level(E)$, for each edge~$E$, as  the unique $j$ such that the uniform mesh~$\hat\cT_j$ contains $E$, and
$\level(T)$, for each simplicial element or face~$T$, as the smallest $j$ such that $T$ can be resolved by $\hat\cT_j$.
Similarly, for the indices~$\lambda = (j,k)\in \Theta$, we define $\abs{\lambda} = j$.

Moreover, we introduce the following notation for Legendre coefficients: for $\nu \in \cF$ and $v \in \cV$,
\[
   [ v ]_\nu =  \int_Y v(y) L_\nu(y)\,\sdd\sigma(y) ,
\]
and similarly, for functionals $\xi \in \cV' = L_2(Y,V',\sigma)$, 
\begin{equation}\label{eq:functionalcoeffs}
   [  \xi ]_\nu = \int_Y \xi(y) L_\nu(y)\,\sdd\sigma(y)   \,,
\end{equation}
so that
\[
   \langle \xi , v \rangle_{\cV',\cV} = \sum_{\nu \in \cF} \langle [\xi]_\nu, [v]_\nu\rangle_{V',V}  .
\]
We also define the stochastic components $[B]_{\nu\nu'}\colon V\to V'$ of the operator $B$ by its 
diffusion coefficient 
\[
[a]_{\nu\nu'}(x)= \int_Y 
a(x,y)
L_\nu(y)  L_{\nu'}(y) 
\sdd\sigma(y) , \] 
and 
\begin{equation}\label{eq: stoch operator components}
  \bigl\langle [B]_{\nu\nu'} v, w \bigr\rangle
  =
  \int_D  
  [a]_{\nu\nu'}(x) \nabla v(x)\cdot \nabla w(x) \sdd x. 
\end{equation}

\subsection{Assumptions}
For stochastic Galerkin methods based on multilevel expansions of random fields, the following assumptions on the functions~$\theta_\mu$  are natural; see \cite{bachmayr2024multilevel} for an overview of expansions of this type and their applications.
\begin{assumption}\label{ass:wavelettheta}
  We assume $\theta_\mu \in L_\infty(D)$ for $\mu \in \cM_0$ such that  the following conditions hold for all $\mu \in \cM$:
  \begin{enumerate}[{\rm(i)}] 
  \item\label{ass:waveletthetaDiam} $\diam \supp \theta_\mu\eqsim 2^{-|\mu|}$, 
  \item\label{ass:waveletthetaLocalization}  $\#\{\mu\in \cM\colon \abs{\mu}=\ell, \supp\theta_\mu\cap B(x,r)\}\lesssim \max\{1, 2^{d\ell}r^d\} $ for all $x\in \R^d$ and $r>0$, where $B(x,r)$ is the ball of radius $r$ with center $x$,
  \item\label{ass:waveletthetaDecay} for some $\alpha>0$, we have $\norm{\theta_\mu}_{L_\infty(D)} \lesssim 2^{-\alpha |\mu|}$.
  \end{enumerate}
\end{assumption}
As a consequence of \Cref{ass:wavelettheta}\eqref{ass:waveletthetaLocalization} we also have the bound $\#\{\mu\in \cM\colon \abs{\mu}=\ell\}\lesssim  2^{d\ell} $.

\begin{remark}
Under similar assumptions, approximation rates for the resulting solutions $u$ have been shown in~\cite{BCDS:17}. In particular, under more general conditions on the $\theta_\mu$ (which are implied by \Cref{ass:wavelettheta}), but additionally assuming $H^2$-regularity of the corresponding Poisson problem, the results in \cite{BCDS:17} yield for $\alpha \in (0,1]$ and $d\geq 2$ the existence of families of meshes $\mathbb T$ with $N(\mathbb T) \to \infty$ such that for any $s < \alpha / d$, there exists $C>0$ such that 
\begin{equation}
   \min_{v \in \cV(\mathbb T) } \norm{u - v}_\cV \leq C N(\mathbb T)^{-s}.
\end{equation}
As the numerical tests in \cite{BV} indicate, with spatial approximations of sufficiently high order, this holds also for $\alpha>1$.
As discussed further also in \cite{BV,bachmayr2024multilevel} and the survey \cite{bachmayr2024multilevel}, the limiting rate $\alpha/d$ is the same rate as the best possible rate for semidiscrete sparse polynomial approximation \cite{BCM:17} and for finite element approximation of a single randomly chosen (spatial) realization of the solution $u(y)$ with $y \sim \sigma$. In the present context, optimality thus means achieving rates arbitrarily close to $\alpha/ d$.
\end{remark}

To arrive at a method with controlled computational complexity, we also require the functions~$\theta_\mu$ to be piecewise polynomial on suitable simplicial meshes as follows.

\begin{assumption}\label{ass:pwpolytheta}
  There exists $m\in\N_0$, $\thetaelementnum\in \N_0$ and $K\in \N$ such that for all $\mu \in \cM$, $\theta_\mu \in \PP_m(\hat\cT_{\abs{\mu}+\thetaelementnum})$ and $\#\{T\in \hat\cT_{\abs{\mu}+k}\colon \supp \theta_\mu \cap T\neq \emptyset\}\leq K$.
\end{assumption}

In what follows, we restrict ourselves to piecewise linear finite elements as in \eqref{eq:defV} for convenience, but higher-order finite elements could be used in an entirely analogous manner. However, in the case of rough input random fields corresponding to $\alpha < 1$ in \Cref{ass:wavelettheta}\eqref{ass:waveletthetaDecay}, piecewise linear elements already yield the best possible convergence rate.

We consider the more general scenario that the diffusion coefficient is given by 
\begin{equation}\label{eeq:NonlinearAffineDiffusion}
  a(x,y) = f\biggl(\sum_{\mu\in \mathcal M} y_\mu\theta_\mu(x)\biggr)
\end{equation}
with an analytic function $f$.
For the approximation analysis we require \Cref{lem:DisjointSupport} and knowledge of the analytic continuation of $f$ to a rectangle in the complex plane.
\begin{assumption}\label{ass:AnalyticContinuation}
 The function $f$ has a uniformly bounded analytic continuation to the open the rectangle
 \[
  \mathcal R_0 = \textstyle\frac{1}{2}\displaystyle(-r_1-r_2,r_1+r_2) + \textstyle\frac{i}{2}\displaystyle(-r_1+r_2, r_1-r_2)
 \]
 with 
 \[ 
 r_1 =  \frac{C\rho}{1-2^{-(\alpha -\alpha')}} 
 \quad\text{and}\quad 
 r_2=\frac{ C\rho^{-1}}{1-2^{-(\alpha +\alpha')}}.
 \]
  Here, $C$ is the hidden constant from \Cref{lem:DisjointSupport}\eqref{ass: coefficient decay}, $\rho>1$, $\alpha'<\alpha$ and $\abs{f(z)}\leq M$ for $z\in \mathcal R_0$.
\end{assumption}

We may assume the constant $C$ in \Cref{ass:AnalyticContinuation} to be one, by replacing the functions $\theta_\mu$ with $\frac{1}{C}\theta_\mu$ and $f(z)$ by $f(Cz)$.


\subsection{Semi-discrete residual approximation}

We now describe the residual approximation in an adaptive solver when in contrast to \cite{BV, bachmayr2024convergent}, the diffusion coefficient is not affine linear as in \eqref{eq:affinecoeff}, but of the form \eqref{eq:loglinearcoeff} or the more general form \eqref{eeq:NonlinearAffineDiffusion}.
To this end,
we first need a notion of costs for an approximated operator.
For an approximation  $\tilde B$ of $B$ we define the costs as
\[
  \cost(\tilde B) = 
  \sup_{ \nu\in \mathcal F}  \#
  \bigl\{ 
    \nu'\in\mathcal F\colon   [\tilde B]_{\nu\nu'} \neq 0 \bigr\}
  =
  \sup_{v\in V, \nu\in \mathcal F} \#
  \supp \bigl([\tilde B(v\otimes L_\nu )]_{\nu'} \bigr)_{\nu'\in\mathcal F}
  .
\]
Then for any $v\in \cV$,
\[
  \#\supp(  [\tilde Bv]_{\nu} )_{\nu\in\cF}
  \leq
  \cost(\tilde B)
  \#\supp(  [v]_{\nu} )_{\nu\in\cF}
\]
For computing residual estimates and for approximately solving Galerkin systems, we require a sequence of such approximate operators $(B_\opApproxnum)_{\opApproxnum\in \N_0}$ satisfying
\begin{equation}\label{eq: approx and cost}
\begin{aligned}
  \norm{B-B_\opApproxnum}_{\mathcal V\to \mathcal V'}&\leq \Ccomp 2^{-\alpha' \opApproxnum},\\
  \cost(B_\opApproxnum)&\leq \Ccost2^{d \opApproxnum}
\end{aligned}
\end{equation} 
with  $\alpha'$ possibly differing from $\alpha$ in \Cref{ass:wavelettheta}\eqref{ass:waveletthetaDecay}.
In the affine linear case~\eqref{eq:affinecoeff} such compressions are given by truncation, as for example used in~\cite{BV, bachmayr2024convergent}
with $\alpha = \alpha'$. This relies on the three term recursion of Legendre polynomials, which yield
\begin{multline}\label{eq:threetermrecursion}
  \int_Y \int_D  y_\mu\theta_\mu  L_\nu(y)\nabla v \cdot L_{\nu'}(y) \nabla w \sdd x \sdd\sigma(y)\\
  =
  \Bigl(\sqrt{\beta_{\nu_\mu +1}}\delta_{\nu+e_\nu,\nu' } + \sqrt{\beta_{\nu_\mu}}\delta_{\nu-e_\nu,\nu' }  \Bigr)\int_D \theta_\mu\nabla v\cdot \nabla w \sdd x
\end{multline}
with $\beta_k = (4 - k^{-2})^{-1}$.

In \Cref{sec:OperatorCompression}, we provide approximations of operators for the log-uniform problem~\eqref{eq:loglinearcoeff} in \Cref{thm: operator compression} and for more general diffusion coefficients in \Cref{thm: polynomial operator compression}, at the price that constants deteriorate for smaller values of $\alpha- \alpha'$. With this machinery at hand, we can derive the same properties as for the affine case treated in~\cite{BV,bachmayr2024convergent} of the routine \textsc{Apply} given in \Cref{alg:apply_semidiscr}.
Under our general assumptions on approximations of operators, we have the following result, where for sequences $\mathbf{a}$ on arbitrary index sets and $s>0$, we use the quasi-norms
\[
\norm{\ba}_{\cA^s} = \max_{N\geq 0} \,(N+1)^s \min_{\#\supp \bz \leq N}\norm{\ba-\bz}_{\ell_2}
\]
that characterize best $N$-term approximability with rate $s$.

\begin{algorithm}[t]
	\caption{$(F_i, \opApproxnum_i)_{i = 0}^I = \text{\textsc{Apply}}(v; \eta)$, for $N\coloneq \#\supp ([v]_\nu)_{\nu \in \cF} < \infty$, $\eta>0$.} \label{alg:apply_semidiscr}
\begin{enumerate}[{\bf(i)}]
\item If $\norm{B}_{\cV\to\cV'} \norm{v}_{\cV} \leq \eta$, return the empty tuple with $F = \emptyset$;
otherwise, with $\bar I=\ceil{\log_2 N}$, for $i = 0,\ldots, \bar I$, determine $F_i \subset \cF$ such that $\# F_i \leq 2^{i}$ and $P_{F_i} v = \sum_{\nu \in F_i} [v]_\nu L_\nu$ satisfies
\begin{equation*}
    \norm{v- P_{F_i} v }_{\cV} \leq C \min_{\# \tilde F \leq 2^i} \norm{ v - P_{\tilde F} v }_\cV  
\end{equation*}
with an absolute constant $C>0$.
Choose $I$ as the minimal integer such that
\[
  \delta = \norm{B}_{\cV\to\cV'} \norm{v - P_{ F_I} v }_\cV \leq \frac{\eta}{2}.
\]

\item \label{algstep:computeapproxnum}With $d_0 =P_{F_0} v$, $d_i = (P_{F_i} - P_{F_{i-1}} )v$, $i=1,\ldots,\bar I$, and $N_i = \# F_i$, set 
\begin{equation*}
	\opApproxnum_i = \Biggl\lceil\frac1{\alpha'} \log_2\biggl( \frac{\Ccomp}{\eta - \delta} \biggl( \frac{\norm{d_i}_\cV}{N_i} \biggr)^{\frac{\alpha'}{\alpha'+d}} \Bigl( \sum_{k=0}^I \norm{d_k}_\cV^{\frac{d}{\alpha'+d}} N_k^{\frac{\alpha'}{\alpha'+d}} \Bigr) \biggr)\Biggr\rceil,
	\quad i = 0, \ldots, I
\end{equation*}
and return $(F_i, \opApproxnum_i)_{i = 0}^I$.
\end{enumerate}
\end{algorithm}

\begin{prop}\label{prop:semidiscrapply}
	Let $s >0$ with $s< \frac{\alpha'}{d}$, let $B$ and the sequence $(B_\opApproxnum)_{\opApproxnum\in\N_0}$ satisfy~\eqref{eq: approx and cost}, let $v$ satisfy $\#\supp ([v]_\nu)_{\nu\in\cF} < \infty$, and let $\opApproxnum_i$ and $d_i$, $i=0,\ldots,I$, be as defined in \Cref{alg:apply_semidiscr}.  Then 
	    \begin{equation}\label{eq:elljest}
      \max_{i=0,\ldots,I} \opApproxnum_i \lesssim 1 + \abs{\log \eta} + \log \bignorm{\bigl(\norm{[v]_\nu}_V \bigr)_{\nu\in\cF}}_{\cA^s}.
    \end{equation}
Moreover, for $g = \sum_{i = 0}^{I} B_{\opApproxnum_i} d_i$ and $F^+ = \supp ([g]_\nu)_{\nu \in \cF}$, we have 
\begin{equation}\label{eq:StochApplyApprox6y}
\norm{B v - g}_{\cV'} \leq \sum_{i = 0}^{I}\norm{(B- B_{\opApproxnum_i}) d_i}_{\cV'}\leq  \eta
\end{equation}
and
    \begin{equation}\label{eq:Fest}
    \# F^+
     \lesssim  \sum_{i=0}^I 2^{d \opApproxnum_i} \# F_i   \lesssim \eta^{-\frac1s} \bignorm{\bigl(\norm{[v]_\nu}_V \bigr)_{\nu\in\cF}}_{\cA^s}^{\frac1s}.
    \end{equation}
  The constants in the inequalities depend on $C$ as in~\Cref{alg:apply_semidiscr}, on $C_B$, $d$, $\alpha'$, $s$, $\Ccomp$, and on $\Ccost$  from~\eqref{eq: approx and cost}.
\end{prop}

\subsection{Adaptive Galerkin method}

\begin{algorithm}[t]
	\caption{Adaptive Galerkin Method}\label{alg:AdaptiveMethod}
  \flushleft  Set the parameters $\opApproxnum, \zeta, \omega_0< \omega_1$ and relative accuracy $\rho$, set initial $u^0= 0$ and formally, $\norm{\hat\br^{-1}}_{\ell_2} = C_\Psi\norm{f}_{\cV'}$. Let $k = 0$;
  \begin{enumerate}[{\bf(i)}]
\item \label{alg:point2}$(\Lambda^+,\hat\br^k,([r^k]_\nu)_\nu, \eta_k, b_k) = \text{\textsc{ResEstimate}}(u^k; \zeta, \frac{\zeta}{1+\zeta}\norm{\hat\br^{k-1}}_{\ell_2}, \varepsilon)$
\item If $b_k\leq \varepsilon$, return $u^k$.
\item\label{algstep:TreeApprox} $\Lambda^{k+1} = \textsc{TreeApprox}(\Lambda^k, \Lambda^+, \hat\br^k, (1-\omega_0^2)\norm{\hat\br^k}^2_{\ell_2})$ satisfying
 \[ 
 \begin{aligned}
    \norm{\hat \br^k|_{\Lambda^{k+1}}}_{\ell_2} &\geq \omega_0 \norm{\hat\br^k}_{\ell_2} ,\\
    \#(\Lambda^{k+1} \setminus \Lambda(\T^k) ) &\lesssim 
    \#(\tilde \Lambda \setminus \Lambda(\T^k) )\quad\text{for any $\tilde \Lambda\supset\Lambda(\T)$ with tree structure}\\ &\text{such that $ \norm{\hat \br^k|_{\tilde\Lambda}}_{\ell_2} \geq \omega_1 \norm{\hat\br^k}_{\ell_2} $;}
 \end{aligned}
 \]
\item\label{algstep:Mesh} $\T^{k+1} = \msx{\Lambda^{k+1}}$;
\item\label{algstep:GalSolve} $u^{k+1} = \text{\textsc{GalerkinSolve}}(\T^{k+1}, u^k , r^k, \opApproxnum, \frac1{\sqrt{c_P}}\frac{\rho}{c_\Psi}\norm{\hat\br^k}_{\ell_2} )$
\item $k\leftarrow k+1$ and go to~{\bf(\ref{alg:point2})};
\end{enumerate}
\end{algorithm}

We now recapitulate the adaptive Galerkin method \Cref{alg:AdaptiveMethod} from~\cite{bachmayr2024convergent}. 
While we follow the standard strategy of adaptive finite element schemes in iteratively refining discretizations by computing error estimates from approximate Galerkin solutions, the particular realization of error estimation and refinement differs in some crucial aspects.

The method \textsc{ResEstimate} (\Cref{alg:ResEstimate}, following~\cite[\alg 3.4]{bachmayr2024convergent}) in \Cref{alg:AdaptiveMethod}\eqref{alg:point2} generates localized residual error estimates and is the most involved routine. It first utilizes \textsc{Apply} in \Cref{alg:apply_semidiscr} to produce a residual approximation $r^k \in \cV'$ with $\supp ([r^k]_\nu)_\nu= F^+\subset\cF$ with 
\[
  \Bigl(\sum_{\nu\in\cF}\bignorm{[f- Bu^k]_\nu - [r^k]_\nu}_{V'}^2 \Bigr)^\frac12\leq \eta,
\]
where $[r^k]_\nu = 0$ for $\nu\notin F$.
Here each $[r^k]_\nu$ is piecewise polynomial on some $\cT_\nu\geq \initmesh$. 
Then for any $J$ and for each $\nu\in F^+$ we can produce a finite set $\Theta^+_\nu\subset\Theta$ 
such that
\[
  \Bigl(\sum_{\lambda\notin\Theta^+_\nu} \bigl\langle [r^k]_\nu, \psi_\lambda \bigr\rangle_{V',V}^2 
  \Bigr)^\frac12
  \leq 
  C2^{-J}
  \Bigl(
  \sum_{\lambda\notin\Theta} \bigl\langle [r^k]_\nu, \psi_\lambda \bigr\rangle_{V',V}^2
  \Bigr)^\frac12
\]
utilizing that $[r^k]_\nu$ is piecewise polynomial.
Then with 
\[
  \Lambda^+ = \bigl\{(\nu,\lambda)\in \cF\times \Theta\colon \nu\in F^+, \lambda\in \Theta^+_\nu  \bigr\}
\]
we set
\[
  \hat\br^k_{\nu,\lambda} = \bigl\langle [r^k]_\nu, \psi_\lambda \bigr\rangle_{L^2(D)}
\]
for $(\nu,\lambda)\in \Lambda^+$, and $ \hat\br^k_{\nu,\lambda}  = 0$ otherwise.
The method \textsc{ResEstimate} now finds a suitable $\Lambda^+$ such that with 
\[
\bz^k_{\nu,\lambda} = \bigl\langle [f_Bu^k]_\nu, \psi_\lambda \bigr\rangle_{V',V}
\]
for any $(\nu,\lambda)\in \cF\times \Theta$, we have the estimate
\[
b_k-\eta_k \leq \norm{\hat\br^k - \bz^k}_{\ell_2(\cF\times \Theta)} \leq \zeta \norm{\bz^k}_{\ell_2(\cF\times \Theta)} \leq b_k.
\]
The method \cite[\alg{3.4}]{bachmayr2024convergent} is tailored to the case of affine coefficients~\eqref{eq:affinecoeff}. We adapt this method in \Cref{alg:ResEstimate} to accommodate the use of more general operator approximations introduced in \Cref{sec:OperatorCompression}.

The discretization refinement is carried out in \Cref{alg:AdaptiveMethod}\eqref{algstep:TreeApprox} and \eqref{algstep:Mesh}  by the methods \textsc{TreeApprox}\cite[\alg 4.4]{BV} and \textsc{Mesh} as described in \cite[\sect{3.2}]{bachmayr2024convergent}, generating a suitable family of meshes $\T^{k+1}$. Finally, in \Cref{alg:AdaptiveMethod}\eqref{algstep:GalSolve} an approximate Galerkin solution in $\cV(\T^{k+1})$ is computed with the method \textsc{GalerkinSolve}\cite[\alg 3.5]{bachmayr2024convergent}.

\subsection{Convergence and optimality}
In~\cite[\theo{4.3}]{bachmayr2024convergent} we have shown the error reduction by a uniform factor of the iterates generated by \Cref{alg:AdaptiveMethod}.
The requirements on the parameters to achieve $\delta<1$ and hence also achieve error reduction are
\begin{equation}\label{eq:zetaCondition}
  0<(C_{B,\Psi}+2)\zeta<1\quad\text{with}\quad C_{B,\Psi} = \frac{C_\Psi^2c_B}{c_\psi^2c_B}
\end{equation}
and
\begin{equation}\label{eq:omega0Condition}
  \omega_0 > (C_{B,\Psi}+1)\frac{\zeta}{1-\zeta}.
\end{equation}
Then it is possible to also find parameters $\opApproxnum$, $\rho$ and $\hat J$ such that $\gamma = \gamma(\zeta, \opApproxnum, \rho, \hat J)$ satisfies
\begin{equation}\label{eq:gammaCondition}
  \gamma  < \frac{\omega_0 - (1+\omega_0)\zeta}{(1+\zeta)\sqrt{C_{B,\Psi}C_B}C_\Psi}.
\end{equation}

The main result of~\cite{bachmayr2024convergent} is the following theorem on error reduction.

\begin{theorem}[{see \cite[\theo{4.3}]{bachmayr2024convergent}}]\label{thm:ErrorReduction}
Let~\eqref{eq:zetaCondition} and~\eqref{eq:omega0Condition} hold.
  Then there exist parameters $\opApproxnum,\rho,\hat J$ such that~\eqref{eq:gammaCondition} is satisfied, and the adaptive algorithm achieves the error reduction
  \begin{equation}\label{eq:ErrorReduction}
    \norm{u-u^{k+1}}_B \leq \delta \norm{u-u^{k}}_B,
  \end{equation}
  where 
  \begin{equation}\label{eq:delta_error_reduction}
    \delta = \biggl(  1  - \frac{\bigl( \omega_0 - (1 + \omega_0)\zeta \bigr)^2 }{ C_{B,\Psi} } + \gamma^2 ( 1 + \zeta )^2 C_\Psi^2 C_B  \biggr)^{\frac12}.
  \end{equation}
\end{theorem}

One of the main results of the present work is that the iterates in \Cref{alg:AdaptiveMethod} are also quasi-optimal approximations to the solution whenever $u \in \mathcal A^s$ with $s > 0$, where $\mathcal A^s \subset \cV$ is the class of solutions that can be approximated at rate $s$ with respect to the total number of degrees of freedom, that is,
\[
   \mathcal A^s = \{ v \in \cV \colon \abs{ v }_{\mathcal A^s } < \infty \} \,,\qquad
      \abs{v}_{\mathcal A^s} 
  =
  \sup_{N > 0}
 N^s 
  \min_{N(\mathbb T)\leq N}
  \min_{v_N\in \mathcal V(\mathbb T)}\norm{v - v_N}_\mathcal V
   < \infty.
\]
Moreover, we show that for $s < \alpha /d$, the computational costs have scaling arbitrarily close to the optimal rate.

We obtain these optimality results under the additional condition
\begin{equation}\label{eq:omega1Condition}
  \omega_1(1-\zeta) +\zeta < (1-2\zeta) {\frac{c_\mathrm{stable}\sqrt{c_B}}{C_\Psi \sqrt{C_B}}},
\end{equation}
where $\cstable$ is a constant only dependent on $\hat\cT_0$ introduced in \Cref{sec: Stability conforming frames}. We prove the following result in \Cref{sec:proofOptimalityResult}.
\begin{theorem}\label{thm: quasioptimal approximation}
  Let the parameters of \Cref{alg:AdaptiveMethod} satisfy \eqref{eq:zetaCondition}, \eqref{eq:omega0Condition}, \eqref{eq:gammaCondition}, and~\eqref{eq:omega1Condition}, and
  let the solution $u$ of~\eqref{eq:diffusionequation} satisfy $u \in \mathcal A^s$.
  Then we have the following:
  \begin{enumerate}[{\upshape(i)}]
  \item \emph{Quasi-optimal approximations to $u$:}
  The iterates $u^k$ of \Cref{alg:AdaptiveMethod} satisfy 
  \begin{equation*}\label{eq:quasi_optimal_approximation}
  N(\mathbb T^k) - 
  N(\mathbb T^0) \leq C_1C_2^\frac1s \bigl(1- \delta^{\frac1s}\bigr)^{-1} \norm{u-u^k}_{\cV}^{-\frac1s} \norm{u}_{\mathcal A^s}^{\frac1s} \,,
  \end{equation*}
  where $C_1>0$ depends only on shape regularity, $C_2\geq1$ is independent of $u$, $k$ and $s$, and $\delta<1$ is given by~\eqref{eq:delta_error_reduction}.
  \item \emph{Optimal computational complexity in the affine case up to logarithmic factor:}
  If the diffusion coefficient is affine linear as in~\eqref{eq:affinecoeff} and Assumptions \ref{ass:wavelettheta} and \ref{ass:pwpolytheta} hold, then if $s<\frac{\alpha}{d}$, 
  the number of elementary operations in the algorithm is bounded by 
  \begin{equation*}\label{eq:NumberOfOperationsAffine}
    C_3 \bigl(1+\abs{\log\norm{u-u^k}_\cV} +\log{\norm{u}_{\cA^s}}\bigr)^3
   \norm{u-u^k}_\cV^{-\frac1{s}} \norm{u}_{\cA^s}^{\frac1{s}} \,,
  \end{equation*}
  with $C_3$ independent of $k$.
  \item \emph{Almost optimal computational complexity in the non-affine case:}
  If the diffusion coefficient is given by~\eqref{eeq:NonlinearAffineDiffusion} and if Assumptions \ref{ass:wavelettheta}, \ref{ass:pwpolytheta}, and \ref{ass:AnalyticContinuation} hold, then there exist sequences of operator approximations $B_\opApproxnum$ so that for any $s'<s<\frac{\alpha}{d}$, 
  the number of elementary operations in the algorithm is bounded by 
  \begin{equation*}\label{eq:NumberOfOperations}
    C_4\bigl(1+\abs{\log\norm{u-u^k}_\cV} +\log{\norm{u}_{\cA^s}}\bigr)^r
    \norm{u-u^k}_\cV^{-\frac1{s'}} \norm{u}_{\cA^s}^{\frac1{s'}} \,,
  \end{equation*}
 where $C_4 > 0$ is independent of $k$, and $r>0$ is independent of $s', s$ and $k$.
\end{enumerate}
\end{theorem}

\section{Compressible infinite elliptic systems of PDEs}\label{sec:OperatorCompression}

As a first first step in the adaptive scheme, we require a semidiscrete adaptive compression of the operator $B\colon \mathcal V\to\mathcal V'$. In the affine linear case~\eqref{eq:affinecoeff} a suitable compression is given by the truncated operators $B_\opApproxnum$ by
\begin{equation}\label{eq: Btrunclinear}
  \langle B_\opApproxnum v , w \rangle = \int_Y \int_D \Bigl(\theta_0 + \sum_{\substack{\mu \in \mathcal{M} \\ \abs{\mu} < \opApproxnum}} y_\mu\theta_\mu \Bigr) \nabla v(y) \cdot \nabla w(y)\sdd x \sdd\sigma(y) \quad \text{for all $v,w \in \cV$}
\end{equation} 
for $\opApproxnum \in \mathbb N$. A similar truncated operator can also be defined for the log-affine case 
\begin{equation}\label{eq: Btrunclog}
  \langle B_\opApproxnum v , w \rangle = \int_Y \int_D \exp\Bigl(\theta_0 + \sum_{\substack{\mu \in \mathcal{M} \\ \abs{\mu} < \opApproxnum}} y_\mu\theta_\mu \Bigr) \nabla v(y) \cdot \nabla w(y)\sdd x \sdd\sigma(y) \quad \text{for all $v,w \in \cV$.}
\end{equation}
In both cases, we get exponential decay with respect to $\opApproxnum$ of the approximation error as long as the coefficients decay themselves, that is,
\begin{equation}
  \label{eq:multilevel2}
    \sum_{|\mu|=\opApproxnum} \abs{\theta_{\mu}} \leq C_2 2^{-\alpha  \opApproxnum} \quad \text{a.e.~in $D$.}
  \end{equation}
\begin{prop}\label{prop:operator compression}
  There is a constant $C_3>0$ depending on $C_2$, $\alpha$, and in the case of~\eqref{eq: Btrunclog} on $C_B$, such that for all $\opApproxnum \geq 0$,
 \[
  \norm{B - B_\opApproxnum}_{\cV\to\cV'} \leq C_3 2^{-\alpha\opApproxnum}.
 \]
\end{prop}
\begin{proof}
  For~\eqref{eq: Btrunclinear} this is~\cite[\propo{3.1}]{bachmayr2024convergent} or~\cite[\propo{3.2}]{BV}. In the case of~\eqref{eq: Btrunclog}, we have an additional factor given by the Lipschitz constant of $\exp$, which equals $C_B$.
\end{proof}

\subsection{Stochastic compression by stochastic truncation}

The truncated operator $B_\opApproxnum$ results from $B$ by replacing $[B]_{\nu\nu'}$ by $0$ if $\nu_\mu\neq \nu'_\mu$ with $\abs\mu\geq \opApproxnum$.
In the affine linear case~\eqref{eq:affinecoeff}, this truncation already suffices, as the stochastic components $[B]_{\nu\nu'}$ are given by the three term recursion~\eqref{eq:threetermrecursion} so that 
only $~2^{d\opApproxnum}$ stochastic components of $B_\opApproxnum$  are nonzero. In the log-affine case~\eqref{eq:loglinearcoeff} however, $[B]_{\nu\nu'}$ vanishes only if $\nu$ and $\nu'$ differs at entries at $\mu$ and $\mu'$ and $\theta_\mu\theta_{\mu'} = 0 $ in $ D$, as will be seen later in \Cref{lem: decay of operator coefficient}. Therefore, applying the truncated operator $B_\opApproxnum$ still leads to infinitely many stochastic components. This makes a more involved compression essential.

We first analyse the situation when only a collection of $A_{\nu\nu'}$ is applied. 
We make use of the following modification of Schur's lemma to estimate the truncation error. 
The proof is given in \Cref{sec:lemmas}.
\begin{lemma}\label{lem: Stoch Schur Lemma}
  Let the operator $B\colon \mathcal V\to \mathcal V'$ be given by
  \[
  \langle Bv, w \rangle = \sum_{\nu\in \mathcal F} \sum_{\nu'\in \mathcal F} \int_D [a]_{\nu\nu'}\nabla [v]_{\nu'}\cdot \nabla [w]_{\nu} \sdd x
  \]
  with $[a]_{\nu\nu'} = [a]_{\nu'\nu}$.
  Then 
  \[
  \norm{B}_{\cV\to \cV'} \leq  \max_{\nu\in \mathcal F}{\Bignorm{\sum_{\nu'\in \mathcal F} \abs{[a]_{\nu\nu'}}}_{L_\infty(D)}}.
  \]
\end{lemma}

We make use of the following approximation result.
\begin{lemma}\label{lem: decay interaction Legendre}
  Let $f:\mathcal E_\rho\to \C$ be holomorphic in the ellipse $\mathcal E_\rho = \bigl\{\frac12(z+z^{-1})\in \C\colon  \abs{z}< \rho \bigr\}$ and continuous up to its boundary. Let $\sigma_1$ be the uniform measure $\sdd \sigma_1(y) =\frac{\sdd y}{2}$. Then  for $k>k'$,
  \[
    \biggabs{\int_{-1}^1 f(y) L_k(y) L_{k'}(y)\sdd \sigma_1(y)} \leq 2 \max_{z\in \mathcal E_\rho}{\abs{f(z)}} \frac{\rho^2}{\rho^2-1} \rho^{-(k-k')}
  \]  
  and 
  \[
    \biggabs{\int_{-1}^1 f(y) L_k(y) L_{k}(y)\sdd \sigma_1(y)} \leq  \max_{z\in [-1,1]}{\abs{f(z)}} .
  \] 
\end{lemma}
\begin{proof}
  Since $f$ is holomorphic in $\mathcal E_\rho$, we can expand $f = \sum_{\ell = 0}^\infty f_\ell T_\ell$ as a sum of Chebyshev polynomials with $\abs{f_\ell} \leq 2\max_{z\in\mathcal E_\rho}\abs{f(z)} \rho^{-\ell}$; see e.g.,
  \cite[\theo{8.1}]{Trefethen:20}. Therefore, we have 
  \begin{align*}
    \biggabs{\int_{-1}^1 f(y) L_k(y) L_{k'}(y)\sdd \sigma_1(y)} 
    &=
    \biggabs{\int_{-1}^1 \sum_{\ell = 0}^\infty f_\ell T_\ell(y) L_k(y) L_{k'}(y)\sdd \sigma_1(y)} \\
    &=  \biggabs{\int_{-1}^1 \sum_{\ell = k-k'}^\infty f_\ell T_\ell(y)L_k(y) L_{k'}(y)\sdd \sigma_1(y)} ,
    \end{align*}
where we have used that $L_k$ is orthogonal to polynomials of degree less than $k$. Moreover, as the integral of an odd function vanishes and by the decay of Chebyshev coefficients, 
    \begin{align*}
      \biggabs{\int_{-1}^1 f(y) L_k(y) L_{k'}(y)\sdd \sigma_1(y)} 
    &=  \biggabs{\int_{-1}^1 \sum_{\ell = 0}^\infty f_{k-k'+2\ell} T_{k-k'+2\ell}(y)L_k(y) L_{k'}(y)\sdd \sigma_1(y)} \\
    &\leq \sum_{\ell = 0}^\infty \abs{f_{k-k'+2\ell}}
    \leq 2 \max_{z\in \mathcal E_\rho}{\abs{f(z)}} \frac{\rho^2}{\rho^2-1} \rho^{-(k-k')}.
  \end{align*}
 To obtain the second statement, we observe that
\[
    \biggabs{\int_{-1}^1 f(y) L_k(y) L_{k}(y)\sdd \sigma_1(y)} \leq
    \max_{z\in [-1,1]}{\abs{f(z)}}\int_{-1}^1 L_k(y)^2 \sdd \sigma_1(y) =  \max_{z\in [-1,1]}{\abs{f(z)}} .  \qedhere
\] 
\end{proof}

For further analysis, we show the following disjointness of supports of the functions $\theta_\mu$.


\begin{lemma}\label{lem:DisjointSupport}
  Let \Cref{ass:wavelettheta} be satisfied. 
  Then there exists a number $\sublev$ dependent on $d$ and the hidden constants in \Cref{ass:wavelettheta} and a level assignment $\abs\cdot_\sublev\colon \cM \to \N_0$ satisfying
  \begin{enumerate}[{\rm(i)}] 
    \item \label{ass: same level disjoint support} if $\abs\mu_\sublev = \abs{\mu'}_\sublev$ and $\mu\neq \mu'$, then $\theta_\mu\theta_{\mu'}=0$ in $L_\infty(D)$, 
    \item \label{ass: diff level disjoint support}
    for all $\ell< \abs\mu_\sublev$ there is at most one $\mu'$ with $\ell = \abs{\mu'}_\sublev$ and $\theta_\mu\theta_{\mu'}\neq 0$ in $L_\infty(D)$,
    \item \label{ass: number of levelfunctions} $\displaystyle\#\{\mu\colon \abs\mu_\sublev = \ell\} \lesssim 2^{\frac{d}{\sublev}\ell}$,
    \item \label{ass: coefficient decay}$\displaystyle\sum_{\abs{\mu}=\ell} \abs{\theta_\mu}_\sublev \lesssim 2^{-\frac{\alpha}{\sublev}\ell}$.
    \end{enumerate}
\end{lemma}
\begin{proof}
  Without loss of generality, assume that 
  $\diam\supp\theta_\mu\leq 2^{-|\mu|}$. Let $r_\ell = 2^{-\ell}\sqrt{d}$. Then 
  \[
  D  =  \bigcup_{x \in 2^{-\ell} \Z^d } B(x,r_\ell )\cap D.
  \]
  Also, for $x,y\in 2^{-\ell} \Z^d$ with $\norm{x-y}\geq 3 r_\ell$ and $\abs{\mu} = \abs{\mu'} = \ell$ with $\supp\theta_\mu\cap B(x, r_j)\neq \emptyset$, $\supp\theta_{\mu'}\cap B(y, r_j)\neq \emptyset$ we have $\theta_\mu\theta_{\mu'} = 0$. This condition is satisfied for $x,y\in 2^{-\ell} (3d\Z)^d$, and thus we can achieve~\eqref{ass: same level disjoint support} with $3^dd^d$ sublevels. The point~\eqref{ass: diff level disjoint support} is also achieved similarly.
\end{proof}

\begin{remark}
  We provide convergence rates that depend only on the ratio $\alpha / d$. Thus, for the asymptotic rate, the value of $\sublev$ in \Cref{lem:DisjointSupport} does not enter. It does, however, enter into the corresponding constants.
\end{remark}
For further approximation and complexity estimation, it makes sense to further categorize the differences $\nu - \nu'$.

\begin{definition}\label{def: level indices}
  Let 
  \[\mathcal F(\N_0) = \{ \mathbf k = (k_\ell)_{\ell\in \N_0}\colon k_\ell\in \N_0 \text{ and } k_\ell = 0 \text{ for all but finitely many $
  \ell \in \N_0$}\}
  \] 
  be the set of sequences with values in the natural numbers with only finitely many nonzero entries. Then we define the set-valued functions
\[
    \mathcal S_\mathbf k(\nu) = \Bigl\{\nu'\in \mathcal F\colon \sum_{\abs{\mu}=\ell} \abs{\nu_\mu-\nu'_\mu} = k_\ell, \prod_{\mu\colon \nu_\mu- \nu'_\mu\neq 0} \theta_\mu \neq 0 \in L_\infty(D)\Bigr\}.
\] Here, the empty product equals one, so that $\mathcal S_\mathbf 0(\nu) = \{\nu \}$.
\end{definition}

\begin{lemma} \label{lem: decay of operator coefficient} Let \Cref{ass:wavelettheta} be satisfied, $\sublev$ be from \Cref{lem:DisjointSupport}, and the diffusion coefficient given by~\eqref{eq:loglinearcoeff}. Then $[a]_{\nu\nu'}\neq 0$ implies $\nu'\in \mathcal S_\mathbf k(\nu)$ for some $\mathbf k \in \mathcal F(\N_0)$. Moreover, for any $\nu\in \mathcal F$,  $0<\alpha'< \alpha$ and $\rho>1$, we get the estimate
\begin{equation}\label{eq:CoefficientSumEstimate}
  \sum_{\nu'\in \mathcal S_\mathbf k(\nu) } \abs{ [a]_{\nu\nu'}} \lesssim  
  4^{\# \{\ell\colon k_\ell > 0\}} \prod_{\ell\in \N_0} (\rho 2^{\frac{\alpha'}{\sublev}\ell})^{-k_\ell}
\end{equation}
in $D$ with constants dependent on $\alpha', \alpha$, $\rho$, and $\sublev$,
and 
\begin{equation}\label{eq: number of coefficients}
\#\mathcal S_\mathbf k(\nu) \lesssim 2^{\# \{\ell\colon k_\ell > 0\}} 2^{\frac{d}{\sublev}\max \{\ell\colon k_\ell > 0\}}.\end{equation}

\end{lemma}
\begin{proof}
Without loss of generality, we assume $\abs{\cdot} = \abs{\cdot}_\sublev$ from \Cref{lem:DisjointSupport}.
Recall that the diffusion coefficient $[a]_{\nu\nu'}$ is given by 
\[
[a]_{\nu\nu'} = \prod_{\mu \in \mathcal M} \int_{-1}^1 \exp(y_\mu\theta_\mu) L_{\nu_\mu}(y_\nu)
L_{\nu'_\mu}(y_\nu) \sdd \sigma_1(y_\mu).
\]
Now let $\nu_\mu-\nu'_\mu\neq 0 $ for some $\mu\in \mathcal M$ and $x\notin \supp \theta_\mu$. By orthogonality of the Legendre polynomials it follows that 
$x\notin \supp [a]_{\nu\nu'}$. As a first consequence,
$[a]_{\nu\nu'}= 0 $ if $\nu'\notin \bigcup_{\mathbf k\in \mathcal F(\N_0)} \mathcal S_\mathbf k(\nu)$ by definition of $\mathcal N_k(\nu)$. Furthermore, we note the following property of $\nu'\in \mathcal S_\mathbf k(\nu)$: let $L = \max \{\ell\colon k_\ell>0 \}$ and let $\nu_\mu-\nu'_\mu\neq 0$ with $\abs{\mu}=L$. Then the other entries of $\nu-\nu'$ are determined up to sign due to \Cref{lem:DisjointSupport}\eqref{ass: diff level disjoint support}. This directly gives \eqref{eq: number of coefficients} due to \Cref{lem:DisjointSupport}\eqref{ass: number of levelfunctions}. Also, different choices of $\mu$ with $\abs{\mu}= L$ lead to disjoint supports due to \Cref{lem:DisjointSupport}\eqref{ass: same level disjoint support}. 

It remains to provide the estimate~\eqref{eq:CoefficientSumEstimate}.
Let $\mu$ with $\abs{\mu} = L$ and $\nu_\mu- \nu'_\mu\neq 0$. Then $[a]_{\nu\nu'}= 0$ on $D\setminus\supp \theta_{\mu}$, and on $\supp \theta_{\mu}$ we have the estimates
\begin{multline*}
  \biggabs{\int_{-1}^1 \exp(y_{\mu'}\theta_{\mu'}) L_{\nu_{\mu'}}(y_{\mu'})
  L_{\nu'_{\mu'}}(y_{\mu'}) \sdd \sigma_1(y_{\mu'})} 
  \\
  \leq \begin{cases}
    1 & \text{if  $\theta_{\mu'}\theta_\mu = 0 $ in $L_\infty(D)$,}\\
    \exp\Bigl(\frac{\theta_{\mu'}}2(\rho_{\mu'}+\rho_{\mu'}^{-1})\Bigr) \frac{\rho_{\mu'}^2}{\rho_{\mu'}^2 -1} 
    &\text{if $\nu_{\mu'} - \nu'_{\mu'} =0$, }\\
    2 \exp\Bigl(\frac{\theta_{\mu'}}2(\rho_{\mu'}+{\rho_{\mu'}^{-1}})\Bigl) \frac{\rho_{\mu'}^2}{\rho_{\mu'}^2 -1} \rho^{-\abs{\nu_{\mu'}-\nu'_{\mu'}}}
    &\text{else, }
  \end{cases}
\end{multline*}
due to \Cref{lem: decay interaction Legendre}.
Setting $\rho_{\mu'} = \rho 2^{\alpha' \abs{\mu'}}$, we get 
\begin{align*}
  \abs{[a]_{\nu\nu'}}
  &\leq
  2^{\# \{\ell\colon k_\ell > 0\}}
  \prod_{\ell = 0}^\infty \exp(c \rho 2^{-(\alpha - \alpha')\ell}) \frac{\rho^2 2^{2\alpha'\ell}}{\rho^2 2^{2\alpha'\ell}-1}(\rho 2^{\alpha'\ell})^{-k_\ell}\\
  &=
  2^{\# \{\ell\colon k_\ell > 0\}} \exp
    \sum_{\ell = 0}^\infty \Bigl(c \rho 2^{-(\alpha - \alpha')\ell} - 
    \ln(1 - \rho^{-2} 2^{-2\alpha'\ell})
  \Bigr)
  \prod_{\ell = 0}^\infty (\rho 2^{\alpha'\ell})^{-k_\ell}
  ,
\end{align*}
where $c$ is the hidden constant in \Cref{lem:DisjointSupport}\eqref{ass: coefficient decay} 
and by Lipschitz continuity of the logarithm,
\[
  \sum_{\ell = 0}^\infty \Bigl(c \rho 2^{-(\alpha - \alpha')\ell} - 
  \ln\bigl(1 - \rho^{-2} 2^{-2\alpha'\ell}\bigr)\Bigr)
  \leq c\rho \frac1{1 - 2^{-(\alpha-\alpha')}} + \frac{1}{(\rho^2-1)(1-2^{-2\alpha'})} .
\]
The estimate~\eqref{eq:CoefficientSumEstimate} follows by summation and taking disjoint supports into account.
\end{proof}

The combination of \Cref{lem: Stoch Schur Lemma} and \Cref{lem: decay of operator coefficient} suggest to approximate the operator $B$ by 
\begin{equation}\label{eq: approx operator}
  \langle B_\mathcal Kv, w\rangle = \sum_{\mathbf k\in \mathcal K}\sum_{\nu\in \mathcal F}\sum_{ \nu'\in \mathcal S_\mathbf k(\nu) } 
  \int_D [a]_{\nu\nu'} \nabla v_{\nu'}\cdot \nabla w_{\nu} \sdd x
\end{equation}
with $\mathcal K \subset \mathcal F(\N_0)$. An application of this operator to $L_\nu(y) v_\nu(x)$ results in a functional consisting of at most
$\sum_{\mathbf k \in \mathcal K} \# \mathcal S_\mathbf k(\nu)$ summands of product Legendre polynomials with coefficients in $V'$. With $Q$ from \Cref{lem:DisjointSupport}, we thus assign to each $\mathbf k\in \mathcal F(\N_0)$ the number
\begin{equation}
\label{eq: cost of coefficient}
\cost(\mathbf k) = \#\mathcal S_\mathbf k(\nu) \lesssim 2^{\# \{\ell\colon k_\ell > 0\}} 2^{\frac{d}{\sublev}\max \{\ell\colon k_\ell > 0\}}.
\end{equation}
The approximation error is controlled by \Cref{lem: Stoch Schur Lemma} and \eqref{eq:CoefficientSumEstimate}. We are now interested in estimates of the decay
\[
\min_{\sum_{\mathbf k\in\mathcal K} \cost(\mathbf k) \leq N} \norm{B-B_\mathcal K}_{\cV \to \cV'}
\]
with respect to $N$. To this end, we use the following variant of Stechkin's Lemma. Its proof is given in \Cref{sec:lemmas}.
\begin{lemma}
\label{lem: Stechkin}
Let $(a_i)_{i\in\N}\subset \R_+$ be nonincreasing, and let $b_i>0$. Then for $0<p<1$,
\[
     \Bigl(\sum_{i =   k+1}^\infty a_i b_i \Bigr)
    \leq
    \Bigl(\sum_{i\in \N} a_i^p b_i\Bigr)^\frac1p
    \Bigl(\sum_{i=1 }^k b_i\Bigr)^{1-\frac1p}.
\] 
\end{lemma}
We want to apply \Cref{lem: Stechkin} to sequences indexed by $\mathcal F(\N_0)$ with certain decay conditions. As an upper bound we use the following lemma.

\begin{lemma}\label{lem: summability of reference sequence}
Let $0<\alpha$, $d\in\N_0$ and  $\rho >1$. Furthermore, for $\mathbf k \in \mathcal F(\N_0)$ let 
\[
a_\mathbf k = 2^{-d\max \{\ell\colon k_\ell > 0\}}
\prod_{\ell\in \N_0} \bigl(\rho 2^{\alpha \ell}\bigr)^{-k_\ell}\quad \text{and}
\quad
b_\mathbf k = 2^{\# \{\ell\colon k_\ell > 0\}}2^{d\max \{ \ell\colon k_\ell > 0\}}
\]
Then for $p>0$ satisfying $\frac 1p -1 < \frac\alpha d$, we have 
\begin{equation}\label{eq: summability of ref sequence}
  \sum_{\mathbf k\in \mathcal F(\N_0)} a_\mathbf k^p b_\mathbf k
  \leq 
  1 + \rho^{-p}
  \exp\biggl(
    \frac{\rho^{-p} (2 -\rho^{-p})}{(1- \rho^{-p})(1-2^{-\alpha  p })}\biggr)
    \frac{1}{1 - 
    2^{( d (1-p) -\alpha p )}}
  <\infty.
\end{equation}
Let $\mathbf k_i$ be an ordering of $\mathcal F(\N_0)$ such that $a_{\mathbf k_i}\geq a_{\mathbf k_{i+1}}$ for all $i \in \N$. 
Then if $d>0$
\begin{equation}\label{eq: cost growth ref sequence}
b_{\mathbf k_{i+1}}\leq 2^d \max_{j\leq i}b_{\mathbf k_{i+1}}.
\end{equation}
\end{lemma}

\begin{proof}
  We treat the factor $2^{d\#\max\{\ell: k_\ell>0\}}$ by introducing the sets
  \[
  \mathcal F(\N_0, L) = \{k\in \mathcal F(\N_0)\colon \text{$k_\ell = 0$ for $\ell>L$}\}
  \]
  and noting that 
  \[
    \mathcal F(\N_0) = \bigcup_{L = 0}^\infty
    \{ \mathbf k + e_L\colon 
      \mathcal F(\N_0, L)\}.
  \]
  We also use the identity
  \[
  \prod_{\ell = 0}^L\biggl( \sum_{k=0}^\infty a_{k,\ell} + \sum_{k=1}^\infty a_{k,\ell} \biggr)
  =
  \sum_{\mathbf k \in \mathcal F(\N_0,L)} 2^{\#\{\ell\colon k_\ell>0\}}\prod_{\ell = 0}^L
  a_{k_\ell,\ell}
  \]
  for general products of sequences.
We get 
  \begin{align*}
  \sum_{\mathbf k\in \mathcal F(\N_0)}
  a_\mathbf k^p b_\mathbf k
  &= a_\mathbf 0^p b_\mathbf 0 + 
  \sum_{L = 0}^\infty
  \sum_{\mathbf k\in \mathcal F(\N_0,L)}
  a_{\mathbf k + e_L}^p b_{\mathbf k + e_L}\\
  &= 1 + 
  \sum_{L = 0}^\infty
 \rho^{-p} 2^{-\alpha L p }
  2^{dL (1-p) }
  \sum_{\mathbf k \in \mathcal F(\N_0,L)}
  2^{\# \{\ell\colon k_\ell > 0\}}
  \prod_{\ell=0}^L (\rho^p 2^{\alpha'\ell p })^{-k_\ell}\\
  &= 1 + 
  \sum_{L = 0}^\infty
  \rho^{-p} 2^{-\alpha L p }
  2^{dL (1-p) }
  \prod_{\ell=0}^L 
  \biggl(
  \sum_{k =0}^\infty
  (\rho^p 2^{\alpha \ell p })^{-k}
  +
  \sum_{k =1}^\infty
  (\rho^p 2^{\alpha \ell p })^{-k}
  \biggr)\\
  &=
  1 + 
  \rho^{-p}
  \sum_{L = 0}^\infty
   2^{-\alpha L p + dL (1-p) }
  \prod_{\ell=0}^L 
  \frac{1+\rho^{-p} 2^{-\alpha \ell p }}{1- \rho^{-p} 2^{-\alpha \ell p }},
  \end{align*}
and the product in the last line is bounded 
  \begin{align*}
    \prod_{\ell=0}^L 
  \frac{1+\rho^{-p} 2^{-\alpha \ell p }}{1- \rho^{-p} 2^{-\alpha \ell p }}
    &\leq
    \exp\biggl(\sum_{\ell=0}^\infty \bigl(
    \ln
    (1+\rho^{-p} 2^{-\alpha \ell p })-\ln(1- \rho^{-p} 2^{-\alpha \ell p })\bigr)\biggr)\\
    &\leq
    \exp\biggl(\sum_{\ell=0}^\infty
    \rho^{-p}\biggl( 1+ \frac{1}{1- \rho^{-p}} \biggr)2^{-\alpha \ell p }\biggr)
    =
    \exp\biggl(
    \frac{\rho^{-p} (2 -\rho^{-p})}{(1- \rho^{-p})(1-2^{-\alpha  p })}\biggr).
  \end{align*}
using Lipschitz continuity of the logarithm in the second inequality.
We arrive at the result
\begin{align*}
  \sum_{\mathbf k\in \mathcal F(\N_0)}
  a_\mathbf k^p b_\mathbf k
  &\leq 1 + \rho^{-p}
  \exp\biggl(
    \frac{\rho^{-p} (2 -\rho^{-p})}{(1- \rho^{-p})(1-2^{-\alpha  p })}\biggr)
    \sum_{L = 0}^\infty
    2^{( d (1-p) -\alpha p ) L}\\
    &=
    1 + \rho^{-p}
  \exp\biggl(
    \frac{\rho^{-p} (2 -\rho^{-p})}{(1- \rho^{-p})(1-2^{-\alpha  p })}\biggr)
    \frac{1}{1 - 
    2^{( d (1-p) -\alpha p )}}<\infty
\end{align*}
for $ d (1-p) -\alpha p <0 $ which is~\eqref{eq: summability of ref sequence}. 

We verify the result~\eqref{eq: cost growth ref sequence} by studying three different cases of sequences $\mathbf k \in \mathcal F(\N_0)$.
Let $L = \max\{\ell\colon k_\ell >0\}$. If $k_L >1$ then 
$a_{\mathbf k -e_L} < a_{\mathbf k}$ and 
$b_{\mathbf k -e_L} = b_{\mathbf k}$  which implies~\eqref{eq: cost growth ref sequence} for 
$\mathbf k_{i+1} = \mathbf k$. 
If $k_L = 1$ and $k_{L-1} = 0$, then 
 $a_{\mathbf k -e_L + e_{L-1}} < a_{\mathbf k}$ and $2^d b_{\mathbf k -e_L + e_{L-1} } = b_{\mathbf k}$ again implying~\eqref{eq: cost growth ref sequence}. The last case is $k_L = 1$ and $k_{L-1} > 0$. Then 
 $a_{\mathbf k - e_{L-1}} < a_{\mathbf k}$ and $2 b_{\mathbf k -e_L + e_{L-1} } \leq b_{\mathbf k}$ and the result follows.
\end{proof}
We are now in the position to certify approximation rates based on $\cost(B_\mathcal K)$ for approximations of $B$.

\begin{theorem}
\label{thm: operator compression} Let \Cref{ass:wavelettheta} be satisfied, let $\sublev$ be as in \Cref{lem:DisjointSupport}, and let the diffusion coefficient be given by~\eqref{eq:loglinearcoeff}. Then 
there is a sequence $(\mathbf k_j)_{j\in \N}\subseteq \mathcal F(\N_0)$ such that the sets $\mathcal K_J = \{\mathbf k_j\colon j\leq J\}$ satisfy
\[
  \norm{B-B_{\mathcal K_J}}
\lesssim 
\biggl(\sum_{j=1}^J \cost(\mathbf k_j) \biggr)^{-s}
\]
for any $s< \frac{\alpha}d$, with constants depending on $s, \alpha$, $d$, and $\sublev$. Moreover, 
\[
  \sum_{j=1}^{J+1} \cost(\mathbf k_j) \lesssim 2^{\frac{d}{\sublev}}
  \sum_{j=1}^{J} \cost(\mathbf k_j)
.\]
\end{theorem}
\begin{proof}
  Without loss of generality we assume $\abs{\cdot} = \abs{\cdot}_\sublev$ from \Cref{lem:DisjointSupport}.
We prove the rate by applying \Cref{lem: Stechkin} to 
sequences of the form described in \Cref{lem: summability of reference sequence}. 
Applying \Cref{lem: decay of operator coefficient} with $\rho  = 2\tilde\rho > 2$ and $\alpha'<\alpha$, for
\[
a_\mathbf k  =  
2^{-d\max \{\ell\colon k_\ell > 0\}}
\prod_{\ell\in \N_0} (\tilde\rho 2^{\alpha'\ell})^{-k_\ell}, \qquad   b_\mathbf k  =  
  2^{\# \{\ell\colon k_\ell > 0\}}2^{d\max \{\ell\colon k_\ell > 0\}}
\]
we have
\[
  \sum_{\nu'\in \mathcal S_\mathbf k(\nu) } \abs{ [a]_{\nu\nu'}}\lesssim a_\mathbf k  b_\mathbf k , \qquad
  \cost(\mathbf k )\lesssim b_\mathbf k. 
\]
We apply \Cref{lem: summability of reference sequence} to get summability of the dominating sequence $a_\mathbf k$ and hence also get summability for 
$p>0$ satisfying for $\frac{1}{p} - 1 < \frac{\alpha'}{d}$
and thus \Cref{lem: Stechkin}

It follows that the sum is finite for $\alpha' p > d(1-p)$, that is, if $\frac1p -1 < \frac{\alpha'}{d}$. Since the choice of $0<\alpha'<\alpha$ is arbitrary, the result follows with \Cref{lem: Stechkin}.
\end{proof}

\subsection{Approximation of diffusion coefficients by polynomials}

Approximation by polynomials plays a central role in the proof of \Cref{lem: decay interaction Legendre}. Indeed, it is possible to get more general approximation results for the stochastic diffusion operators by directly approximating the diffusion coefficient $a$ by polynomials in the stochastic variables. We first consider the stochastic storage costs of applying a diffusion operator that is polynomial in its stochastic variables.
\begin{lemma}\label{lem: polynomial approx legendre interaction}
Let $\tilde a(x,y)$ be of the form 
\[
\tilde a(x,y) = \sum_{\nu\in \mathcal K} a_\nu(x) y^\nu
\]
where $\mathcal K\subset \mathcal F$ with the property 
 if $\nu \in \mathcal K$ and $\nu'_\mu\leq \nu_\mu$ for all $\mu\in\mathcal M$, then $\nu'\in\mathcal K$ as well.
Then with $\tilde B\colon \mathcal V\to \mathcal V'$ given by 
\[
\langle \tilde B v, w \rangle =
\int_D \int_Y \tilde a(x,y)\nabla_x v(x,y)\cdot \nabla_x w(x,y) \sdd\sigma(y) \sdd x,
\]
for any $v \in V$ and $\nu \in \cF$, we have $\tilde B (v\otimes L_\nu) = \sum_{\nu'\in \mathcal K'} g_{\nu'}\otimes L_{\nu'}$
with certain $g_{\nu'}\in V'$, $\nu ' \in \mathcal{K}'$, where
\[
\#\mathcal K'\leq \sum_{\nu\in \mathcal K} 2^{\#\{\mu\colon \nu_\mu>0\}}.
\]
\end{lemma}
\begin{proof}
First we note that 
\[
  y^{\nu'} L_\nu(y) = \prod_{\mu\in\cM} y_\mu^{\nu'_\mu} L_{\nu_\mu}(y_\mu)
  = \prod_{\mu\in\cM} \Bigl(\sum_{k_\mu = \nu_\mu - \nu'_\mu}^{\nu_\mu + \nu'_\mu} 
  c_{k_\mu} L_{k_\mu}(y_\mu)\Bigr)
  = \sum_{\mathbf k\in \mathcal K_{\nu'}} \prod_{\mu\in\cM}  c_{k_\mu} L_{k_\mu}(y_\mu)
\]
with $\mathcal K_{\nu'} = \{\mathbf k\in \cF\colon \nu_\mu - \nu'_\mu\leq k_\mu\leq \nu_\mu + \nu'_\mu \}$. Moreover, 
\begin{align*}
\mathcal K_{\nu'} \setminus \Bigl( \bigcup_{\mu\in \cM\colon \nu_\mu\geq 1} \mathcal K_{\nu'-e_\mu}\Bigr) = \{\mathbf k\in \cF\colon k_\mu = \nu_\mu - \nu'_\mu \text{ or } k_\mu  = \nu_\mu + \nu'_\mu \},\\
\# \{\mathbf k\in \cF\colon k_\mu = \nu_\mu - \nu'_\mu \text{ or } k_\mu  = \nu_\mu + \nu'_\mu \} \leq 2^{\#\{\mu\colon \nu'_\mu>0\}}.
\end{align*}
The result follows by the downward closed property of $\mathcal K$.
\end{proof}

The error analysis is simpler as in \Cref{lem: Stoch Schur Lemma}. Let $a(x,y)$ be the original diffusion coefficient and $\tilde a(x,y)$ be an approximation by polynomials in the stochastic variables. Let $B$ and $\tilde B$ be the corresponding operators. Then 
\[
\norm{B-\tilde B}_{\cV \to \cV'} \leq \sup_{x\in D, y\in Y} \abs{a(x,y)-\tilde a(x,y)}.
\]
In the context of polynomials, we consider 
\begin{equation}\label{eq: level indices}
  \mathcal S_\mathbf k(\mathbf 0)=  \Bigl\{\nu\in \mathcal F\colon \sum_{\abs{\mu}=\ell} \nu_\mu = k_\ell, \prod_{\mu\colon \nu_\mu\neq 0} \theta_\mu \neq 0 \in L_\infty(D)\Bigr\},
\end{equation}
see \Cref{def: level indices}.
\begin{theorem}\label{thm: polynomial operator compression}
  Let Assumptions \ref{ass:wavelettheta} and \ref{ass:AnalyticContinuation} hold and let $\sublev$ be as in \Cref{lem:DisjointSupport}.
  Then 
  there is a sequence $(\mathbf k_\jnum)_{\jnum\in \N}\subseteq \mathcal F(\N_0)$ of multi-indices and a sequence of polynomials in the stochastic variables of the form 
  \[
  a_\Jnum(x,y) = \sum_{j = 0}^\Jnum \sum_{\nu \in \mathcal N_{\mathbf k_j}(\mathbf 0)} a^\Jnum_\nu(x) y^\nu
  \]
  such that the sets $\mathcal K_\Jnum = \{\mathbf k_\jnum\colon \jnum\leq \Jnum\}$ satisfy
  \[
    \sup_{x\in D, y\in Y} \abs{a(x,y)- a_\Jnum(x,y)}
  \lesssim 
  \biggl(\sum_{\jnum=1}^\Jnum \cost(\mathbf k_j) \biggr)^{-s}
  \]
  for any $s< \frac{\alpha'}d$ with constants depending on $s, \alpha, \alpha', d$, and $\sublev$ and 
  \[
    \sum_{\jnum=1}^{\Jnum+1} \cost(\mathbf k_\jnum) \lesssim 2^{\frac{d}{\sublev}}
    \sum_{\jnum=1}^{\Jnum} \cost(\mathbf k_\jnum)
  .\]
\end{theorem}


\begin{proof}
  Without loss of generality we assume $\abs{\cdot} = \abs{\cdot}_\sublev$ from \Cref{lem:DisjointSupport}.
As a first step we analyse the expansion of 
\[
f\Bigl(\sum_{\ell = 0}^\infty z_\ell 2^{-\alpha\ell}\Bigr) = \sum_{\mathbf k \in \mathcal{\N_0}} f_\mathbf k \prod_{\ell = 0}^\infty T_{k_\ell}(z_\ell)
\]
by product Chebyshev polynomials. Here $T_n$ is the $n$-th Chebyshev polynomial of first kind. First, as $\abs{T_n(z)}\leq 1$ for $z\in [-1,1]$ we have the estimate 
\begin{equation}  \label{eq: polynomial approximation}
  \biggnorm{f \Bigl(\sum_{\ell = 0}^\infty z_\ell 2^{-\alpha \ell}\Bigr) - \sum_{\mathbf k \in \mathcal K} f_\mathbf k\prod_{\ell = 0}^\infty T_{k_\ell}(z_\ell)}_{L_\infty([-1,1]^{\N_0})}
  \leq 
  \sum_{\mathbf k \in \mathcal F(\N_0)\setminus\mathcal K} \abs{f_\mathbf k}.  
\end{equation}
In order to define the coefficients $f_\mathbf k$ we require the correct measure. 
Let $\sdd \lambda_1(z) =  \frac{\sdd z}{\pi\sqrt{1-z^2}}$ and $\lambda$ the according product measure on $Z = [-1,1]^{\N_0}$. Then 
\[
f_\mathbf k 
= 
2^{\#\{\ell\colon k_\ell >0\}}
\int_Z\prod^\infty_{\ell = 0}T_{k_\ell}(z_\ell) 
f\Bigl(\sum_{\ell = 0}^\infty z_\ell 2^{-\alpha \ell}\Bigr) \sdd \lambda(z)
\]
We can estimate the coefficients recursively. To this end, we define the functions 
\[
f_{\mathbf k, L}(\tilde z) =
2^{\#\{\ell\leq L\colon k_\ell >0\}}
\int_{[-1,1]^{L+1}}
\prod_{\ell = 0}^LT_{k_\ell}(z_\ell) 
f\Bigl(\tilde z+ \sum_{\ell = 0}^\infty z_\ell 2^{-\alpha \ell}\Bigr) \sdd \lambda(z). 
\]
Note that for $k_L >0$,
\[
  f_{\mathbf k, L+1}(\tilde z) 
  = \int_{-1}^1 T_{k_{L}}(z)
  f_{\mathbf k, L}(\tilde z +2^{-\alpha L}z)
  2\sdd \lambda_1(z),
\]
 and for $k_L = 0$,
\[
  f_{\mathbf k, L+1}(\tilde z) 
  = \int_{-1}^1 
  f_{\mathbf k, L}(\tilde z +2^{-\alpha L}z)
  \sdd \lambda_1(z).
\]
We use the bounds of $f$ in $\mathcal R_0$ to deduce bounds of the functions $f_{\mathbf k, L}.$
To this end, we define the rectangles
\[
  \mathcal R_\ell = \frac{1}{2}\Bigl(-r^{(\ell)}_1-r^{(\ell)}_2,r^{(\ell)}_1+r^{(\ell)}_2\Bigr) + \frac{i}{2}\Bigl(-r^{(\ell)}_1+r^{(\ell)}_2, r^{(\ell)}_1-r^{(\ell)}_2\Bigr)
\]
with 
\[ 
  r^{(\ell)}_1 =2^{-(\alpha-\alpha')\ell} r_1  =\rho \frac{2^{-(\alpha-\alpha')\ell}}{1-2^{-(\alpha -\alpha')}} 
  \quad\text{and}\quad 
  r^{(\ell)}_2= 2^{-(\alpha+\alpha')\ell}r_2 = \rho^{-1}\frac{2^{-(\alpha'+\alpha)\ell}}{1-2^{-(\alpha +\alpha')}}
\]
and the ellipses
\[
  \mathcal E_\ell = \biggl\{\frac{2^{- \alpha\ell}}2(z+z^{-1})\colon 1 < \abs{z}\leq \rho 2^{\alpha'\ell}\biggr\}.
\]
Note that $\mathcal R_{\ell+1} + \mathcal E_\ell\subset \mathcal R_\ell$. Hence, using approximation rates by polynomials for analytic functions,  if $k_0 >0$ we obtain
\[
\abs{f_{\mathbf k, 0}(z)}\leq 2M \rho^{-k_0}\quad\text{for all $z\in \mathcal R_1$,}
\]
 and if $k_0 =0$,
\[
\abs{f_{\mathbf k, 0}(z)}\leq M \quad\text{for all $z\in \mathcal R_1$;}
\]
 see e.g., \cite[Theorem 8.1]{Trefethen:20}. Furthermore, the function $a_{\mathbf k, 0}$ is analytic in the interior of $\mathcal R_1$. By induction, we get
\[
\abs{f_{\mathbf k, L}(z)}\leq 2^{\#\{\ell\leq L\colon k_\ell >0\}}M
\prod_{\ell = 0}^L2^{-\alpha'\ell k_\ell}\rho^{-k_\ell}
\]
for $z\in \mathcal R_{L+1}$, where $f_{\mathbf k, L}$ is  analytic. It follows that
\begin{equation}\label{eq:f_kDecay}
\abs{f_\mathbf k}\leq 2^{\#\{\ell\colon k_\ell >0\}}M
\prod_{\ell = 0}^\infty2^{-\alpha'\ell k_\ell}\rho^{-k_\ell}.
\end{equation}
Since $2^{\alpha\ell}\sum_{\abs{\mu}=\ell}\abs{\theta_\mu(x)}\leq 1$ by \Cref{lem:DisjointSupport}\eqref{ass: coefficient decay}, we get
\[
  \biggnorm{f\Bigl(\sum_{\ell = 0}^\infty \sum_{\mu\in \mathcal M}y_\mu\theta_\mu\Bigr) - \sum_{\mathbf k \in \mathcal K} f_\mathbf k\prod_{\ell = 0}^\infty T_{k_\ell}\bigl(2^{\alpha\ell}\sum_{\abs{\mu} =\ell}y_\mu\theta_\mu\bigr)}_{L_\infty(D)}
  \leq 
  \sum_{\mathbf k \in \mathcal F(\N_0)\setminus\mathcal K} \abs{f_\mathbf k}
\]
for all $y\in Y$. By the disjoint support condition \eqref{ass: same level disjoint support} we have 
\[
  T_{k_\ell}\Bigl(2^{\alpha\ell}\sum_{\abs{\mu} =\ell}y_\mu\theta_\mu\Bigr)
  =
  T_{k_\ell}(0) + \sum_{\abs{\mu} =\ell}\bigl(T_{k_\ell}(2^{\alpha\ell}y_\mu\theta_\mu) -T_{k_\ell}(0)\bigr),
\]
that is, the stochastic variables can be separated. Furthermore, 
  $T_{k_\ell}(2^{\alpha\ell}y_\mu\theta_\mu) -T_{k_\ell}(0)$ is a polynomial of degree $k_\ell$ in $y_\mu$ with function-valued coefficients that have the same support as $\theta_\mu$. Thus, by condition~\eqref{ass: diff level disjoint support}, the product can be written as
  \[
    \prod_{\ell = 0}^L T_{k_\ell}\Bigl(2^{\alpha\ell}\sum_{\abs{\mu} =\ell}y_\mu\theta_\mu\Bigr)
    = \sum_{\abs\mu = L} 
    \prod_{\ell = 0}^L T_{k_\ell}\bigl(2^{\alpha\ell }
    y_{A(\mu, \ell)}
    \theta_{A(\mu, \ell)}\bigr),
  \]
  where $A(\mu, \ell)$ is the unique $\mu'\in \mathcal M$ with $\abs{\mu'}=\ell$ such that $\theta_\mu\theta_{\mu'}\neq 0$ if it exists, and an arbitrarily chosen one with $\abs{\mu'}=\ell$, otherwise. For each $\mathbf k\in \mathcal F(\N_0)$ there is $L\in \N_0$ such that $k_\ell=0$ for all $\ell>L$. Thus, each $\mathbf k$ contributes 
  $\#\{\mu\colon \abs\mu = L\} \lesssim 2^{dL}$ polynomials in the stochastic variables where $L = \max\{\ell\colon k_\ell>0\}$. 
  
  Observe that for $\mathbf k, \mathbf k'\in \mathcal N$ with $k'_\ell \leq k_\ell$ for all $\ell$ we have that the estimate in~\eqref{eq:f_kDecay} for $\abs{f_{\mathbf k}}$ is smaller than the estimate for $\abs{f_{\mathbf k'}}$. 
  We may thus assume the set $\mathcal N$ to have the property that if $\mathbf k\in \mathcal N$ then $\mathbf k'\in\mathcal N$ as well. 
  Then the set $\mathcal N\subset \mathcal F(\N_0)$ leads to a polynomial approximation of the form in \Cref{lem: polynomial approx legendre interaction} with 
  \[
  \mathcal K = \bigcup_{\mathbf k \in\mathcal N}  \Bigl\{ \nu\in \mathcal F\colon  \sum_{\abs{\mu}=\ell} \abs{\nu_\mu} = k_\ell, \;\exists x\in D \text{ s.t.} \!\prod_{\mu\colon \nu_\mu\neq 0} \theta_\mu(x) \neq 0 \Bigr\}.
  \]
  Each set in the union contains $\#\{\mu\colon \abs\mu = L\} \lesssim 2^{dL}$ elements, where again $L =  \max\{\ell\colon k_\ell>0\}$. In light of \Cref{lem: polynomial approx legendre interaction}, we may thus associate to each $\mathbf k \in \mathcal F(\N_0)$ the costs 
  \[
  \cost(\mathbf k) = 2^{\#\{\ell\colon k_\ell>0\}}\#\{\mu\colon \abs\mu = \max\{\ell\colon k_\ell>0\}\} \lesssim 2^{\#\{\ell\colon k_\ell>0\} +d \max\{\ell\colon k_\ell>0\}} =b_\mathbf k
  \]
  and then apply \Cref{lem: summability of reference sequence} with $a_\mathbf k = \frac{f_\mathbf k}{b_\mathbf k}$. The result follows. 
\end{proof}

\begin{remark}\label{rem:polynomialCoefficients}
  The proof of \Cref{thm: polynomial operator compression} shows that the spatial coefficients of the approximation $a_\Jnum(x,y)$
  can be chosen as polynomials in the functions $\theta_\mu$.
  They satisfy
  \begin{equation}
    \sum_{\jnum = 0}^\Jnum \sum_{\nu \in \mathcal N_{\mathbf k_\jnum}(\mathbf 0)} a^\Jnum_\nu(x) y^\nu
    =
    \sum_{j = 0}^\Jnum
    \sum_{\mathbf k \in  \mathcal N_{\mathbf k_\jnum}(\mathbf 0)} f_\mathbf k\prod_{\ell = 0}^\infty T_{k_\ell}\bigl(2^{\alpha\ell}\sum_{\abs{\mu} =\ell}y_\mu\theta_\mu\bigr).
  \end{equation}
\end{remark}

\section{Complexity estimates}\label{sec:Complexity}

In the case of the affine linear parametrization~\eqref{eq:affinecoeff}, we have analysed the computational complexity of the procedures that constitute the adaptive method in detail \cite{bachmayr2024convergent}. The main result is \cite[\propo 3.7]{bachmayr2024convergent} for evaluating residual estimators, which are subsequently used to refine the meshes. It states that the number of residual estimators and number of operations scale up to log factor linear in the current degrees of freedom, and in the required precision according to the approximation class of the current approximation of the solution.

In the case of the log-affine linear parametrization~\eqref{eq:loglinearcoeff}, we have a few additional difficulties. 
If we approximate the operator as in \Cref{thm: operator compression}, we have the function valued coefficients
\[
[a]_{\nu\nu'} = \int_Y \exp\Bigl(\sum_{\mu\in \mathcal M} y_\mu \theta_\mu \Bigr)L_\nu(y) L_{\nu'}(y)\sdd\sigma(y).
\]
These can in principle be evaluated exactly, but may require~${\abs{\nu}+\abs{\nu'}}$ integrations by parts and is therefore expensive. Another theoretical problem is that one main building block for the residual estimation is~\cite[\lem{2.2}]{bachmayr2024convergent}. For convenience, we restate it here.

\begin{lemma}\label{lem:spatialrelerr}
  Let $m_1, m_2 \in \N_0$.
  Then there is $J_0 \in \N$ depending only on $\initmesh$ and $C(m_1,m_2)$ with additional polynomial dependence on $m_1, m_2$ such that for all $J\geq J_0$ the following holds:
For any $\cT \geq \initmesh$ with interior facets $\mathcal{E}$ and any $\xi \in V'$ of the form 
\begin{equation}\label{eq:xirepresentation}
     \langle \xi, v\rangle = \sum_{K \in \mathcal{T}} \int_K \xi_{K} v \sdd x + \sum_{E \in \mathcal{E}}  \int_E \xi_E v \sdd s ,\quad v \in V,
\end{equation}
where $\xi_K \in \PP_{m_1}(K)$, $K \in \mathcal{T}$, and $\xi_E \in \PP_{m_2}(E)$, $E \in \mathcal{E}$, we have
\[
    \Bigl( \sum_{\lambda \in \Theta_{J}(\cT)} \bigabs{\langle \xi, \psi_\lambda\rangle_{V',V} }^2   \Bigr)^{\frac12} \leq 
    C(m_1,m_2) 2^{-J}  \Bigl(  \sum_{\lambda \in \Theta}  \bigabs{\langle \xi, \psi_\lambda\rangle_{V',V} }^2 \Bigr)^{\frac12} 
\]
with
\begin{multline}\label{eq:ThetaJdef}
  \Theta_{J}(\cT) = \bigl\{  \lambda \in\Theta \colon  (\forall K \in \mathcal{T}\colon    \meas_2(\supp \psi_\lambda \cap K)> 0 \implies \abs{\lambda} > \level(K) + J )   \\
  \wedge   (\forall E \in \mathcal{E}\colon    \meas_1(\supp \psi_\lambda \cap E)>0 \implies \abs{\lambda} > \level(E) + 2J ) 
  \bigr\} .
\end{multline}
\end{lemma}
\begin{proof}
  This lemma differs from the version~\cite[\lem{2.2}]{bachmayr2024convergent} in the dependence of the polynomial degrees $m_1,m_2$; the proof given there involves the estimates
    \[
    \sum_{\lambda \in \Theta_{J}(\cT)} \bigabs{\langle \xi, \psi_\lambda\rangle_{V',V} }^2\lesssim 2^{-2J} \sum_{K \in \mathcal{T}} h_K^2 \norm{\xi_K}_{L_2(K)}^2 
      + 2^{-2J} \sum_{E \in \mathcal{E}} h_E \norm{\xi_E}_{L_2(E)}^2 \,
      \lesssim 2^{-2J}\norm{\xi}_{V'}^2
 \]
  with the first hidden constant depends on~$\initmesh$ and the second also depending on the polynomial degrees~$m_1, m_2$. 
  Here, the dependence is indeed polynomial in~$m_1, m_2$, as can be seen by~\cite[\propo{3.46}]{V:13} and the proof of~\cite[\theo{3.59}]{V:13}. By~\eqref{eq:PsiNormEquivalence}, it follows that
  \[
   \norm{\xi}_{V'}^2
      \lesssim \sum_{\lambda \in \Theta}  \bigabs{\langle \xi, \psi_\lambda\rangle_{V',V} }^2
  \]
   with a further constant depending on~$\initmesh$ concluding the proof.
\end{proof}
Since the functions $[a]_{\nu\nu'}$ are not polynomials, \Cref{lem:spatialrelerr} is not directly applicable. When we approximate in the spirit of \Cref{thm: polynomial operator compression}, we have polynomial functions in $\theta_\nu$, which again are piecewise polynomials, and \Cref{lem:spatialrelerr} can be applied. However, in contrast to the case of affine parametrizations, we do not have a uniform bound on the arising polynomial degrees. 
The following lemma describes the necessary size of the sets $\#(\Theta\setminus\Theta_J)$.

\begin{lemma}\label{lem:AccuracyCosts}
Under the assumptions of \Cref{lem:spatialrelerr}, for $q>0$, let $J\geq J_0$ be the smallest integer such that $C(m_1,m_2) 2^{-J} \leq q$.
Then with $C(m_1,m_2)$ from \Cref{lem:spatialrelerr} and an additional constant only depending on $\initmesh$,
\[
  \#(\Theta\setminus\Theta_J) \lesssim C(m_1,m_2)^d \#\cT  \frac{1}{q^d}.
\]
\end{lemma}
\begin{proof}
The result follows directly from $\#(\Theta\setminus\Theta_J)\lesssim 2^{dJ}\#\cT $ with hidden constant only depending on $\initmesh$ and the fact that the smallest possible $J$ satisfies $2^{J-1} \leq \frac{C(m_1,m_2)}{q} $.
\end{proof}

The next hurdle is the efficient evaluation of the stochastic integrals. To this end, let $\mu(\ell, x)$ denote the unique index $\mu$ with $\abs{\mu}_\sublev = \ell$ and $\theta_\mu(x)\neq 0$ if it exists, and an arbitrary index of level $\ell$ otherwise.  Then
\begin{multline}\label{eq:ChebyshevLegendreproduct}
\int_Y \prod_{\ell = 0}^\infty T_{k_\ell}\bigl(2^{\alpha\ell}\sum_{\abs\mu=\ell}y_\mu\theta_\mu(x)\bigr)L_\nu(y)L_{\nu'}(y)\sdd \sigma(y)
\\=
\prod_{\ell = 0}^\infty \int_{-1}^1 T_{k_\ell}(2^{\alpha\ell} y_\ell\theta_{\mu(\ell,x)}(x)) L_{\nu_{\mu(\ell,x)}}(y_\ell) L_{\nu'_{\mu(\ell,x)}}(y_\ell)\sdd \sigma_1(y_\ell).
\end{multline}
For small $\nu_{\mu(\ell,x)},\nu_{\mu(\ell,x)}$ and $k_\ell$ this can be computed cheaply by quadrature. It is however also possible to evaluate~\eqref{eq:ChebyshevLegendreproduct} with a complexity of~$O(\sum_{\ell}k_\ell^3)$, independently of $\nu, \nu'$. To achieve this, for $y,z \in \R$ one can iterate the recursion formulas
\begin{equation*}
  \begin{aligned}
    T_{k+1}(z) &= 2zT_k(z) - T_{k-1}(z), \\
    \sqrt{\beta_{k+1}}L_{k+1}(y) &= yL_k(y) -\sqrt{\beta_k}L_{k-1}(y)
  \end{aligned}
\end{equation*}
for $k>1$ and utilize orthonormality of the Legendre polynomials. This leads to the identity
\begin{equation}\label{eq:LegendreChebyshevEval}
\int_{-1}^1 T_k(yz)L_{\ell}(y) L_{\ell'}(y)\sdd \sigma_1(y) 
= \bigl(T_k(Z) \bigr)_{k+1, k+1+\ell'-\ell}
\end{equation}
{with}
\[
Z = z
\begin{pNiceMatrix}
  0  & \!\!\!\sqrt{\beta_{\ell-k}}\!\!\! &  &  &  \\
  \!\!\!\sqrt{\beta_{\ell-k}}\!\!\! & 0 & \Ddots &   & \\
   & \Ddots &  \Ddots & \Ddots  & \\
   &  &  \Ddots & 0 &  \!\!\!\sqrt{\beta_{\ell+k}} \!\!\! \\
   &  &  &  \!\!\!\sqrt{\beta_{\ell+k}} \!\!\! & 0
\end{pNiceMatrix}.
\]

It is thus possible to evaluate~\eqref{eq:ChebyshevLegendreproduct} pointwise with polynomial costs which in turn can be used to compute spatial integrals via quadrature. The arising polynomial degrees can again be bounded in terms of $\mathbf{k}$ and the polynomial degree of the piecewise functions $\theta_\mu$.
For further analysis of the computational complexity, we use the following adaptation of \Cref{lem: summability of reference sequence}. Its proof mirrors the one of \Cref{lem: summability of reference sequence} and is given in \Cref{sec:lemmas} for the sake of completeness.

\begin{lemma}\label{lem:SummabilityOfRefSequence2}
  Let $\alpha>0$, $t,d_a,d_b\in\N_0$, and  $\rho >1$. Furthermore, for $\mathbf k \in \mathcal F(\N_0)$, let 
  \[
  a_\mathbf k = 2^{-d_a\max \{\ell\colon k_\ell > 0\}}
  \prod_{\ell\in \N_0} (\rho 2^{\alpha \ell})^{-k_\ell}\quad \text{and}
  \quad
  b_\mathbf k = 2^{\# \{\ell\colon k_\ell > 0\}}2^{d_b\max \{ \ell\colon k_\ell > 0\}}\prod_{\ell\in \N_0} \binom{k_\ell + t}{t}
  \]
  Then for $p>0$ such that $\frac{d_b}{p} -d_a < \alpha$, we have 
  \begin{equation}\label{eq:SummabilityOfRefSequence2}
    \sum_{\mathbf k\in \mathcal F(\N_0)} a_\mathbf k^p b_\mathbf k
    \leq 
    1 + \rho^{-p}
    \exp\biggl(
      \frac{\rho^{-p}(t+1) (2 -\rho^{-p})}{(1- \rho^{-p})(1-2^{-\alpha  p })}\biggr)
      \frac{1}{1 - 
      2^{d_b-(\alpha+d_a)p}}
    <\infty.
  \end{equation}
\end{lemma}

\begin{lemma}\label{lem:ApplicationCost}
  Let \Cref{ass:wavelettheta} and \Cref{ass:pwpolytheta} be satisfied and let $\sublev$ be as in \Cref{lem:DisjointSupport}.
  For $\opApproxnum\in\N_0$, let $\mathcal K_\opApproxnum\subset \cF(\N_0)$ be a finite downward-closed set. We define
  \[
  L_\mathbf k =\max\{\ell\colon k_\ell>0\},
  \quad
  L = \max_{\mathbf k \in \mathcal K_\opApproxnum} L_\mathbf k,
  \quad
  M_\mathbf k  =  m\sum_{\ell = 0}^L k_\ell,
  \quad 
  M = \max_{k\in\mathcal K_I} M_\mathbf k,
  \quad
  \hat k_\ell = \max_{\mathbf k \in\mathcal K_\opApproxnum} k_\ell.
  \]
  Let $B_\opApproxnum\colon \cV\to \cV'$ be given by
  \[
  \langle B_\opApproxnum v, w \rangle_{\cV',\cV} = \sum_{\mathbf k \in \mathcal K_\opApproxnum} f_\mathbf k  
  \int_D
  \int_Y \prod_{\ell = 0}^L T_{k_\ell}\bigl(2^{\frac\alpha\sublev\ell}\sum_{\abs\mu_\sublev=\ell}y_\mu\theta_\mu\bigr)
  \nabla v \cdot \nabla w  
  \sdd \sigma(y)
  \sdd x\,.
  \]
  For $\nu\in\cF$ and $[v]_\nu\in V(\cT_\nu)$, 
  let $g_\nu = B_\opApproxnum ([v]_\nu \otimes L_\nu) \in \cV'$. 
  Then $\supp ([g_\nu]_{\nu'})_{\nu'\in\cF}\subseteq\bigcup_{\mathbf k\in \mathcal K_\opApproxnum}\mathcal S_{\mathbf k}(\nu)$ and the functional $[g_\nu]_{\nu'}$ is uniquely determined by the piecewise polynomial flux term
  \begin{equation}\label{eq:gkChar}
   \gknuprimepol(x) = 
   \sum_{\mathbf k \in \mathcal K_\opApproxnum} f_\mathbf k  
   \int_Y \prod_{\ell = 0}^L T_{k_\ell}\bigl(2^{\frac\alpha\sublev\ell}\sum_{\abs\mu_\sublev=\ell}y_\mu\theta_\mu(x)\bigr)L_\nu(y)L_{\nu'}(y) 
  \sdd \sigma(y)
  \nabla [v]_\nu(x) 
  \end{equation} 
  in at most
  $\binom{d+M}{M} \sum_{\mathbf k\in \mathcal K_\opApproxnum} 2^{\# \{\ell\colon k_\ell > 0\}}(\#\cT_\nu + 2^{\frac{d}{\sublev} L_\mathbf k}) $
  suitable interpolation points, and
  all $\gknuprimepol$, $\nu\in\bigcup_{\mathbf{k}'\leq \mathbf k}\mathcal S_{\mathbf k'}(\nu)$,
  can be evaluated at each point simultaneously using $O(L\sum_{\mathbf k\in \mathcal K_\opApproxnum} 2^{\# \{\ell\colon k_\ell > 0\}}+ \sum_{\ell = 0}^L\hat k_\ell^3)$ arithmetic operations.

  Moreover, let 
\[
  a_\mathbf k = 2^{-\frac{d}{\sublev}\max \{\ell\colon k_\ell > 0\}}
  \prod_{\ell\in \N_0} (\rho 2^{\frac{\alpha'}{\sublev} \ell})^{-k_\ell}\quad \text{and}
  \quad
  b_\mathbf k = 2^{\# \{\ell\colon k_\ell > 0\}}2^{\frac{d}{\sublev}\max \{ \ell\colon k_\ell > 0\}},
  \]
  and let the set $\mathcal K_I$ satisfy $a_\mathbf k \geq a_{\mathbf k'}$ for all $\mathbf k \in \mathcal{K}_\opApproxnum, \mathbf k'\notin \mathcal K_\opApproxnum$, and $\sum_{\mathbf k\in \mathcal K_\opApproxnum}b_\mathbf k \leq 2^{d\opApproxnum}$. Then 
  \begin{equation}\label{eq:LandMestimate}
  L \leq \sublev\opApproxnum-1,\quad M\leq 1 + \frac{\alpha'+d }{\log_2{\rho}} \opApproxnum
  \end{equation}  
  and for any $p>0$,
  \begin{equation}\label{eq:sumk0estimate}
  \sum_{\mathbf k\in \mathcal K_\opApproxnum} 2^{\# \{\ell\colon k_\ell > 0\}} 
  \leq 
   2^{\opApproxnum{(\alpha'+d)p}}
  \rho^{p}
  \biggl(1 + \rho^{-p}
    \exp\biggl(
      \frac{\rho^{-p} (2 -\rho^{-p})}{(1- \rho^{-p})(1-2^{-{\alpha'}{\sublev}^{-1}  p })}\biggr)
      \frac{1}{1 - 
      2^{-\frac{\alpha'+d}{\sublev}p}}\biggr)
  .
  \end{equation}
\end{lemma}

\begin{proof}
  By \Cref{ass:pwpolytheta} and \Cref{lem:DisjointSupport}, the function $\gknuprimepol|_T$ is a polynomial of degree at most $M$ on each element~$T$ in the smallest common refinement $\cT$ of $\hat\cT_{\ceil{L/\sublev} +\thetaelementnum}$ and $\cT_\nu$. Thus, we need to collect at most $\binom{d+M}{M}$ interpolation points on each element to uniquely determine each $\gknuprimepol$.
  

  Now let~$\mu(\ell, x)$ denote the unique index~$\mu$ with $\abs{\mu}_\sublev = \ell$ and $\theta_\mu(x)\neq 0$. Then for $k_\ell\leq \hat k_\ell$ and $\nu'_{\mu(\ell, x)}  = \nu_{\mu(\ell, x)} \pm k_\ell$,
\[
  \gknuprimepol(x) =
  \sum_{\mathbf k \in \mathcal K_\opApproxnum} f_\mathbf k  
  \prod_{\ell = 0}^L \int_{-1}^1 T_{k_\ell}\bigl(2^{\frac\alpha\sublev\ell}y_\ell\theta_{\mu(\ell,x)}(x)\bigr)L_{\nu_\mu(\ell, x)}(y_\ell)L_{\nu'_\mu(\ell, x)}(y_\ell) 
  \sdd \sigma_1(y_\ell)
  \nabla [v]_\nu(x) 
\]
and $\gknuprimepol(x) = 0$ for all other $\nu'$. 
Therefore, for each element $T\in\cT$ there are $2^{\# \{\ell\colon k_\ell > 0\}}$ indices $\nu\in\cF$ such that $\gknuprimepol|_T\neq 0$. We therefore require at most $\binom{d+M}{M} \sum_{\mathbf k\in \mathcal K_\opApproxnum} 2^{\# \{\ell\colon k_\ell > 0\}}(\#\cT_\nu + 2^{\frac{d}{\sublev} L_\mathbf k}) $ interpolation points.

We can compute the integrals simultaneously via~\eqref{eq:LegendreChebyshevEval}. This results in~$O(\sum_{\ell = 0}^L\hat k_\ell^3)$ basic arithmetic operations per interpolation point since each integral requires $\hat k_\ell$ matrix multiplications with a tridiagonal matrix of size $2\hat k_\ell+1$ and $2 \hat k_\ell$ summations of matrices for the recursion of the Chebyshev polynomials.

The estimate~\eqref{eq:LandMestimate} for $L$ follows directly from $\sum_{\mathbf k\in \mathcal K_\opApproxnum}b_\mathbf k \leq 2^{d\opApproxnum}$. Furthermore, note that $Me_0 \in \mathcal K_\opApproxnum$ and $e_{\sublev\opApproxnum}\notin \mathcal K_\opApproxnum$ and thus
  \[
    \rho^{-M} =  a_{Me_0}\geq a_{e_{\sublev\opApproxnum}} = \rho^{-1} 2^{-{(\alpha'+d)}\opApproxnum},
  \]
which directly implies the bound on $M$. Finally, we apply \Cref{lem:SummabilityOfRefSequence2} to get 
\[
\sum_{\mathbf k\in \mathcal F(\N_0)} a_\mathbf k^p 2^{\# \{\ell\colon k_\ell > 0\}}  
\leq 
1 + \rho^{-p}
\exp\biggl(
      \frac{\rho^{-p} (2 -\rho^{-p})}{(1- \rho^{-p})(1-2^{-{\alpha'}{\sublev}^{-1}  p })}\biggr)
  \frac{1}{1 - 
  2^{-\frac{\alpha'+d}{\sublev}p}}.
\]
Since $ a_\mathbf k \geq \rho^{-1} 2^{-{(\alpha'+d)}\opApproxnum}$ for all $\mathbf k \in \mathcal K_\opApproxnum$, the estimate~\eqref{eq:sumk0estimate} follows readily.
\end{proof}
\begin{remark}
It can be advantageous to not store~\eqref{eq:gkChar} directly. but to store the $L$ factors independently to reuse them for several polynomials.
\end{remark}

\begin{algorithm}[t]
  \caption{$((\Theta_\nu^+)_{\nu\in F^+},(\hat{\br}_\nu)_{\nu\in F^+},([r] _\nu)_{\nu\in F^+}, \eta, b) \!=\! \text{\textsc{ResEstimate}}(v; \zeta, \eta_0, \varepsilon)$, for $\#\supp ([v]_\nu)_\nu < \infty$, relative tolerance $\zeta>0$, initial tolerance $\eta_0$, target tolerance $\varepsilon$.}\label{alg:ResEstimate}

  \flushleft Set $\eta = \eta_0$; choose $q < \zeta$.
  \begin{enumerate}[{\bf(i)}]
    \item Set $(F_i, \opApproxnum_i)_{i = 0}^I = \text{\textsc{Apply}}(v; \eta/C_\Psi)$ by \Cref{alg:apply_semidiscr};
    
  \item  For each $i = 0,\ldots, I$
  \begin{enumerate}[{\bf(1)}]
    \item Set $F_i^+ = \bigcup_{\nu'\in F_i\setminus F_{i-1},\mathbf k \in\mathcal K_{\opApproxnum_i}} \mathcal S_\mathbf k(\nu')$;
    \item For each $\nu' \in F_i\setminus F_{i-1}$ and
     each $\nu \in \supp \bigl(\bigl[B_{\opApproxnum_\nu} [v]_{\nu'}\otimes L_{\nu'}\bigr]_\nu\bigr)_{\nu\in \cF}$ compute 

    \quad\quad $\displaystyle  [g_i]_{\nu,\nu'} =  \bigl[B_{\opApproxnum_\nu} [v]_{\nu'}\otimes L_{\nu'}\bigr]_\nu$;
    

    \item For each $\nu\in F^+_i$ evaluate
  
    \quad\quad $\displaystyle [r_i] _{\nu} = - \sum_{\nu'} [g_i]_{\nu,\nu'} = \text{\textsc{Sum}}\bigl( (-[g_i]_{\nu,\nu'})_{\nu'}\bigr)$;


  \end{enumerate}

    \item Set $F^+ = \bigcup_{i=0}^I F_i^+$;
    
    \item\label{algstep:rnusum} For each $\nu\in F^+$ evaluate

    \quad\quad $\displaystyle [r] _{\nu} =\delta_{0,\nu} f - \sum_{i = 0}^I [r_i]_{\nu} = \text{\textsc{Sum}}\bigl(\delta_{0,\nu} f, (-[g_i]_{\nu,\nu'})_{i = 0}^I\bigr)$

    \noindent
    and set $\hat J_\nu$ so that $ 2^{-\hat J_\nu} C(M-1,M) \leq q$ with $M$ the polynomial degree of $ [r]_{\nu}$;

    \item\label{algstep:DualNormIndicators} For each $\nu\in F^+$ set $(\Theta_{\nu}^+,\hat\br_{\nu}) = \text{\textsc{DualNormIndicators}}( [r]_{\nu},\hat J_\nu)$;
    \item Let $b = \eta + (1+\frac{q}{\sqrt{1-q^2}})\norm{(\norm{\hat\br_\nu}_{\ell_2})_{\nu\in F^+}}$. If $\eta \leq \frac{\zeta - q}{ 1+\zeta} \norm{(\norm{\hat\br_\nu}_{\ell_2})_{\nu\in F^+}}$ or $b \leq \varepsilon$,

    \noindent
    \quad return $((\Theta_\nu^+)_{\nu\in F^+},(\hat\br_\nu)_{\nu\in F^+},([r] _\nu)_{\nu \in F^+}, \eta, b)$; 
    
    \noindent otherwise, set $\eta \leftarrow \eta/2$ and go to  {\bf(i)};
  \end{enumerate}
  \label{optscheme}
  \end{algorithm}
 
We are now in the position to characterize the complexity of \Cref{alg:ResEstimate}.

\begin{prop}\label{prop:ResEstimate}
    Let $((\Theta_\nu^+)_{\nu \in F^+}, (\hat\br_
    \nu)_{
      \nu\in F^+
    }, \eta, b)$ be the return values of Algorithm \ref{optscheme} and let
\begin{equation}\label{eq:Lambdaplus}
  \Lambda^+ = \bigl\{  (\nu, \lambda) \in \cF \times \Theta\colon \nu \in F^+, \lambda \in \Theta_\nu^+ \bigr\}\,.
\end{equation} 
    Set $\hat\br_\nu = 0$ for $\nu\notin F^+$.
    Furthermore, let $\bz = (\langle [f-Bv]_\nu, \psi_\lambda\rangle)_{\nu\in\cF,\lambda\in\Theta}$.
    Then $\norm{\bz}_{\ell_2} \leq  b$ and either $b \leq \varepsilon$, or $\hat\br$ satisfies 
    \begin{equation}\label{eq:reserrbound} 
      \norm{\hat\br -\bz}_{\ell_2} \leq \zeta \norm{\bz}_{\ell_2},
    \end{equation}
    where we have
    $ \#\supp\hat\br \leq \#\Lambda^+ = \sum_{\nu \in F^+} \# \Theta_\nu^+$. Moreover, for any $p>0$, there are polynomials $p_1, p_2$ with degrees only depending on the degrees of the 
    polynomial constant in \Cref{lem:AccuracyCosts}
    and with coefficients depending
    on $p$, on the polynomial constant in \Cref{lem:AccuracyCosts} and on the parameters in \Cref{prop:semidiscrapply}, such that 
    \begin{multline}\label{eq:rsuppest}
     \#\Lambda^+ \leq
     p_1(1+\abs{\log\eta} +\log{\bignorm{\bigl(\norm{[v]_\nu}_V \bigr)_{\nu\in\cF}}_{\cA^s}})
     \\
     \bigl(
     \#\cT(f) + \eta^{-p} \bignorm{\bigl(\norm{[v]_\nu}_V \bigr)_{\nu\in\cF}}_{\cA^s}^{p} N(\T) + \eta^{-\frac1s} \bignorm{\bigl(\norm{[v]_\nu}_V \bigr)_{\nu\in\cF}}_{\cA^s}^{\frac1s}\bigr)\,.
    \end{multline} 
    The number of operations in \Cref{alg:ResEstimate} is bounded by 
    \begin{multline}\label{eq:ropest}
      (1+\log_2(\eta_0/\eta))p_2(1+\abs{\log\eta} +\log{\bignorm{\bigl(\norm{[v]_\nu}_V \bigr)_{\nu\in\cF}}_{\cA^s}})
      \\
      \bigl(
        \#F\log\#F +
      \#\cT(f) + \eta^{-p} \bignorm{\bigl(\norm{[v]_\nu}_V \bigr)_{\nu\in\cF}}_{\cA^s}^{p} N(\T) + \eta^{-\frac1s} \bignorm{\bigl(\norm{[v]_\nu}_V \bigr)_{\nu\in\cF}}_{\cA^s}^{\frac1s}\bigr).
    \end{multline}
    \end{prop}
    \begin{proof}
      The estimate~\eqref{eq:reserrbound} follows in complete analogy to the proof of \cite[\propo{3.7}]{bachmayr2024convergent}.

      For the size of $\Lambda^+$, we first analyse the family of triangulations 
      $\T^+ = (\cT_\nu^+)_{\nu\in F^+}$ so that~$[r]_\nu$ has a piecewise polynomial representation on~$\cT_\nu^+$. Here $T^+_\nu$ does not need to be a complete mesh for $D$. Its size can be estimated with \Cref{prop:semidiscrapply} and \Cref{lem:ApplicationCost} with any $p>0$
      \begin{align*}
      N(\T^+)&\leq \cT(f) + \sum_{i=0}^I\sum_{\nu\in F_i\setminus F_{i-1}}\sum_{\mathbf k \in \mathcal K_{\opApproxnum_i}} 2^{\# \{\ell\colon k_\ell > 0\}} (\# T_\nu +2^\frac{d}{\sublev}L_\mathbf k)\\
      &\lesssim  \cT(f) + 2^{\max_{i=0,\ldots, I}\opApproxnum_i p} N(\T) + \sum_{i=0}^I 2^{d\opApproxnum_i} \#F_i\\
      &\lesssim \cT(f) + \eta^{-p} \bignorm{\bigl(\norm{[v]_\nu}_V \bigr)_{\nu\in\cF}}_{\cA^s}^{p} N(\T) + \eta^{-\frac1s} \bignorm{\bigl(\norm{[v]_\nu}_V \bigr)_{\nu\in\cF}}_{\cA^s}^{\frac1s}
      \end{align*}
      with constants having the same dependency as in \Cref{prop:semidiscrapply} and additionally on $p$.
      Now we estimate the size of $\Lambda^+$. First, we have to complete each $\cT^+_\nu$ to a mesh for $D$ to apply \textsc{DualNormIndicators} in step~\eqref{algstep:DualNormIndicators}. This leads to another factor $L$ as in~\eqref{eq:LandMestimate}. Using \Cref{lem:AccuracyCosts}, we get with $M$ as in~\eqref{eq:LandMestimate}
      \[
      \#\Lambda^+ \lesssim  L C(M-1,M)^dN(\T^+).
      \]
      With~\eqref{eq:elljest}, \eqref{eq:LandMestimate}, and the previous estimate, we get~\eqref{eq:rsuppest}.

      The proof for \eqref{eq:ropest} is again in complete analogy to the proof of \cite[\propo{3.7}]{bachmayr2024convergent} with additional polynomial factors in $L,M$ to account for varying polynomial degrees in \textsc{DualNormIndicators}  and \textsc{Sum} and to compute the required interpolation points as in \Cref{lem:ApplicationCost}.
\end{proof}
\begin{remark}
  In the affine case~\eqref{eq:affinecoeff} it is possible to choose $p = 0$ and the polynomial degree of $p_1$ and $p_2$ can be bounded by two. This is given in equations~(3.11) and~(3.12) in~\cite{bachmayr2024convergent}. 
\end{remark}

Finally we want to note that the coefficients of the product Chebyshev polynomials are readily available by
\[
\exp\Bigl(\sum_{\ell=0}^\infty 2^{-\alpha \ell}z_\ell\Bigr) 
= \prod_{\ell = 0}^\infty \exp(2^{-\alpha \ell}z_\ell)
= \prod_{\ell = 0}^\infty \Bigl(I_0(2^{-\alpha \ell})T_0(z_\ell) + 2\sum_{k = 1}^\infty I_k(2^{-\alpha \ell}) T_k(z_\ell)
 \Bigr)
\]
where $I_k$ is the $k$-th modified Bessel function of first kind.


\section{Optimality of the solver}\label{sec:Optimality}
This section is dedicated to the proof of \Cref{thm: quasioptimal approximation}. 
The key ingredients for the proof are the ellipticity of the selfadjoint operator  $B\colon \mathcal V\to \mathcal V'$
\begin{equation}\label{eq:BnormEquivalence}
c_B\norm{v}^2_\mathcal V
\leq
\langle Bv, v\rangle = \norm{v}_B^2
\leq
C_B\norm{v}^2_\mathcal V\quad\text{for all $v\in \mathcal V,$}
\end{equation}
a frame $\Psi = \{\psi_\lambda\in \mathcal V\colon \lambda\in \Theta\}$ with 
\begin{equation}\label{eq:PsiNormEquivalence}
c_\Psi^2\norm{f}^2_{\mathcal V'}
\leq
\norm{\Psi f}^2_{\ell_2}
=
\sum_{\lambda\in \Theta} \abs{\langle f, \psi_\lambda \rangle}^2
\leq
C_\Psi^2\norm{f}^2_{\mathcal V'}\quad\text{for all $f\in \mathcal V',$}
\end{equation}
and a class of stable admissible subsets $\mathfrak{S}\subset 2^\Theta$ that have the following property. 
\begin{definition}\label{def: stable class}
  A class of subsets $\mathfrak{S}\subset 2^\Theta$ is 
  admissible if $S_1, S_2\in \mathfrak{S}$ implies $S_1\cup S_2\in \mathfrak{S}$ and 
  \[
    \bigcup_{S\in \mathfrak{S}} S = \Theta.
  \] 
  It is 
  stable with constant $\cstable$ if for every $S\in \mathfrak S$ and $v\in V(S) = \Span\{\psi_\lambda\colon \lambda \in S\}$ there is a coefficient vector $\bz\in \ell_2(S) $ with $\sum_{\lambda\in S} z_\lambda \psi_\lambda$ and 
  \[
  \cstable \norm{\bz}_{\ell_2} \leq \norm{v}_\mathcal V.
  \]
\end{definition}
\Cref{sec: Stability conforming frames} is dedicated to show that a class connected to conforming triangulations is stable, and thus making the following more general results applicable to the setting of finite elements.

We can now state an abstract result on the optimality of approximation sets. A version of this with wavelets with tree structure is~\cite[\lem 5.3]{BV} and similar result without additional structure is~\cite[\lem 2.1]{GHS:07}.

\begin{lemma}\label{lem:AbstractOptimality}
Let $\mathfrak{S}$ be an admissible and stable class of subsets with constant $\cstable$ and $Bu =f$. 
Let $\sqrt{C_B}\beta<1$, $\omega\leq \frac{\cstable \sqrt{c_B}}{C_\Psi \sqrt{C_B}}(1-C_B\beta^2)^{\frac{1}{2}}$ 
and $w\in \cV(\Lambda_0)$. 
Then the smallest set $\Lambda\in \mathfrak{S}$ with $\Lambda_0\subset\Lambda$ such that 
\[
\norm{\Psi(Bw-f)|_\Lambda}_{\ell_2} \geq \omega \norm{\Psi(Bw-f)}_{\ell_2} 
\]
satisfies
\[
\#(\Lambda\setminus\Lambda_0) \leq \#(\hat\Lambda\setminus\Lambda_0) 
\]
for all $\hat\Lambda\in \mathfrak{S}$ with 
\begin{equation}\label{eq:betaApproxSetProperty}
\min_{v\in \cV(\hat\Lambda)}\norm{u-v}_\mathcal V\leq \beta \norm{u-w}_B.
\end{equation}
\end{lemma}
\begin{proof}
  Let $\hat\Lambda\in \mathfrak{S}$ with \eqref{eq:betaApproxSetProperty}. Then the Galerkin solution $\tilde u$ in $\cV(\hat\Lambda\cup \Lambda_0)$ satisfies
  \[
    \norm{u-\tilde u}_B 
    \leq \min_{v\in \cV(\hat\Lambda)}\norm{u-v}_B
    \leq \sqrt{C_B} \min_{v\in \cV(\hat\Lambda)}\norm{u-v}_\mathcal V
    \leq \sqrt{C_B} \beta \norm{u-w}_B,
  \]
  and thus by orthogonality
  \[
    \norm{\tilde u - w}^2_B \geq (1 - C_B\beta^2) \norm{u-w}^2_B.
  \]
  We can now find stable coefficients $\bz\in\ell_2(\hat\Lambda\cup \Lambda_0)$ of $\tilde u - w$, i.e., $\cstable\norm{\bz}_{\ell_2} \leq \norm{\tilde u -w}_\cV$. Then 
  \[
    \norm{\Psi(Bw-f)|_{\hat\Lambda}}_{\ell_2}
    = 
    \norm{\Psi B(w-\tilde u)|_{\hat\Lambda}}_{\ell_2}
    \geq
    \cstable
    \norm{\Psi B(w-\tilde u)|_{\hat\Lambda}}_{\ell_2}
    \frac{\norm{\bz}_{\ell_2}}{\norm{\tilde{u} -w}_\mathcal V}
  \]
  and
  using the Cauchy-Schwarz inequality and the norm estimates~\eqref{eq:BnormEquivalence} and~\eqref{eq:PsiNormEquivalence} 
  \[
    \norm{\Psi B(w-\tilde u)|_{\hat\Lambda}}_{\ell_2}
    \frac{\norm{\bz}_{\ell_2}}{\norm{\tilde{u} -w}_\mathcal V}
    \geq
    \frac{\norm{\tilde u - w}^2_B}{\norm{\tilde{u} -w}_\mathcal V}
    \geq \frac{\sqrt{c_B}}{C_\Psi \sqrt{C_B}}  \norm{\Psi( Bw-f)}_{\mathcal V'}.
  \]
  Together this results in 
\[
    \norm{\Psi(Bw-f)|_{\hat\Lambda}}_{\ell_2}
    \geq \omega \norm{\Psi( Bw-f)}_{\mathcal V'}
\]
  and thus by minimality of the size of $\Lambda$, the statement follows.
\end{proof}
\subsection{Stability of conforming subsets of finite element frames}\label{sec: Stability conforming frames}

The previous framework is applicable to finite elements. 

\begin{definition}\label{def: conforming fem frame}
A subset $S\subset \Theta$ is called conforming if there exists a conforming simplicial mesh $\mathcal T$ generated by bisection such that 
\[
S = \{\lambda\in \Theta\colon \psi_\lambda \in V(\mathcal T)\},
\]
and we write $S= S(\mathcal T)$ when $S$ and $\mathcal T$ satisfy this relation.
\end{definition}

It is important to note that conforming subsets of $\Theta$ generate the same finite element space as the corresponding triangulation. 

\begin{lemma} Let $\mathcal T$ be a conforming simplicial mesh. Then $V(\mathcal T) = V(S(\mathcal T))$.
\end{lemma}
\begin{proof} 
  We can form a sequence of conforming meshes $\initmesh = \cT_0\leq \cT_1\leq \ldots \leq \cT_N = \cT$ where $\cT_{k}$ is generated from $\cT_{k-1}$ by bisection of one edge. Thus, $V(\cT_{k}) = V(\cT_{k-1}) \oplus \Span\{\psi_{x_k}\}$ where $\psi_{x_k}$ is the nodal basis function at the newly created vertex $x_k$. Notice that all supporting simplices of $\psi_{x_k}$ were generated by the same number of bisections. Therefore, there is a uniform refinement of $\initmesh$ such that $\psi_{x_k}$ is a nodal basis function of that mesh. Moreover, in the next $d-1$ uniform refinement steps no edge containing $x_k$ is bisected. Thus, there is $j$ such that  $\psi_{x_k}$ is a nodal basis function of $\hat \cT_j$ as $\hat\cT_j$ is generated from $\hat\cT_{j-1}$ by $d$ uniform refinements. We get $ \Span\{\psi_{x_k}\} = \Span\{\psi_{j,k}\}$ and with this 
  $V(\mathcal T) \subseteq V(S(\mathcal T))$. The other inclusion $V(\mathcal T) \supseteq V(S(\mathcal T))$ follows directly from the definition of $S(\mathcal T)$.
\end{proof}

This section is devoted to show that conforming simplicial meshes indeed give rise to stable subsets of finite element frames.

\begin{theorem}\label{thm: conforming is stable}
The conforming subsets $\{S(\mathcal T)\colon \text{$\mathcal T$ is conforming}\}$ form an admissible stable class of subsets with constant $\cstable$ dependent on uniform shape regularity.
\end{theorem}

For the proof of this result, a few properties of uniform refinements are crucial. Let $Q_j$ be the $L_2$-orthogonal projection onto $V(\hat{\mathcal T}_j)$. Then 
\begin{equation}\label{eq: H1 equiv multilevel}
    \norm{v}_{H_0^1(D)}^2
    \eqsim \sum_{j = 0}^\infty 2^{2j} \norm{Q_j v - Q_{j-1} v}_{L_2(D)}^2
    \eqsim \sum_{j = 0}^\infty 2^{2j} \norm{Q_j v - v}_{L_2(D)}^2;
\end{equation}
see e.g. \cite{BY:93} for $d=2$ or \cite[\theo{3.2}]{DK:92}.

Another consequence of uniform refinements is the Riesz-property 
\begin{equation}\label{eq:levelRieszProperty}
    \norm{u_j}_{L_2(D)} \eqsim 2^{-j} \norm{\bz}_{\ell_2}
\end{equation}
where $u_j = \sum_{k\in\cN_j}z_k \psi_{j,k}\in V(\hat\cT_j)$.

\begin{lemma}\label{lem:L2normTriangle}
  Let $f\colon \R^d\to\R$ be a linear function and $T\subset\R^d$ be a $d$ dimensional simplex. Let $v\in\R^{d+1}$ be the values of $f$ on the vertices of $T$. Then there are constants $c_d$ and $C_d$ independent on the choice of $T$ such that
  \[
      c_d \vol(T) \norm{v}_2^2 \leq \norm{f}^2_{L_2(T)} \leq C_d \vol(T) \norm{v}_2^2.
  \] 
  Furthermore, let $f_0$ a linear function where some values on the vertices have been replaced by $0$. Then 
  \[
      \norm{f}^2_{L_2(T)} \geq \frac{c_d}{C_d} \norm{f_0}^2_{L_2(T)}.
  \]  
\end{lemma}
\begin{proof}
  First, let $\lambda_T(x)\in \R^{d+1}$ be the barycentric coordinates of $x$ with respect to the vertices of $T$. Then 
  \[
      f(x) = \lambda_T(x)^\top v
  \]
  and thus
  \[
      \norm{f}^2_{L_2(T)} = v^\top\biggl( \int_T \lambda_T(x)\lambda_T(x)^\top  dx\biggr) v.
  \]
  Now let $\hat T$ be a reference simplex and $L\hat T + n  = T$ an affine linear transformation. Then
  \begin{align*}
      \norm{f}^2_{L_2(T)} &= v^\top\biggl( \int_T \lambda_T(x)\lambda_T(x)^\top  dx\biggr) v\\
      &=
      \abs{\det L}
      v^\top\biggl( \int_{\hat T} \lambda_T(Ly+n)\lambda_T(Ly+n)^\top  dy\biggr) v\\
      &= \frac{\vol(T)}{\vol(\hat T)}v^\top\biggl( \int_{\hat T} \lambda_{\hat T}(y)\lambda_{\hat T}(y)^\top  dy\biggr) v
      .
  \end{align*}
  Now the matrix $M = \int_{\hat T} \lambda_{\hat T}(y)\lambda_{\hat T}(y)^\top  dy$ is independent of $T$ and positive definite. The first part of the result follows with $c_d = \norm{M^{-1}}_2/ \vol(\hat T)$
  and $C_d = \norm{M}_2/ \vol(\hat T)$. The second part follows by 
  \[
      \norm{f}^2_{L_2(T)} \geq c_d \vol(T) \norm{v}^2_2 \geq c_d \vol(T) \norm{v_0}^2_2
      \geq \frac{c_d}{C_d} \norm{f_0}^2_{L_2(T)}    
  \]
  where $v_0$ are the values of $f_0$ on the vertices. 
\end{proof}

The following construction of interpolation operators is crucial for the proof of \Cref{thm: conforming is stable}.

\begin{lemma}\label{lem:QuasiL2Projection}
    Let $\cT\geq \hat\cT_0$ be a conforming triangulation, $j\in \N_0$, and $v\in V(\mathcal T)$.
    Then there is a projection 
    $I_j\colon V(\cT)\to V(\cT)\cap V(\hat\cT_j) $ such that 
    \begin{equation}\label{eq:QuasiL2ProjectionEll}
        ( I_j v -v)_{|T} = 0 \quad\text{for all $T\in \mathcal T$ with $\mathrm{level} (T)\leq j-\ell$}  
    \end{equation}
    and 
    \begin{equation}\label{eq:QuasiL2Projection}
        \norm{ I_j v -v}_{L_2(D)}^2 \leq c_\ell\norm{ Q_j v -v}_{L_2(D)}^2,
    \end{equation}
    where $\ell$ and $c_\ell$ only depend on the shape regularity and are independent of $j$.
\end{lemma}
\begin{proof}
    Let $\tilde\cT$ be the finest mesh such that $\tilde\cT\leq \cT$ and $\tilde\cT\leq \hat\cT_j$.
    Then  $V(\cT)\cap V(\hat\cT_j)= V(\tilde\cT)$.
    Let $x_k$ be the vertices of $\tilde\cT$. We then define the projection $I_j v \in V(\tilde\cT)$ uniquely by its values on each of these vertices
    \[
        (I_j v)(x_k) = \begin{cases}
        (Q_jv)(x_k) & \text{if $x_k$ is a vertex of a simplex $T\in \hat\cT_j\cap \tilde\cT,$  }\\
        v(x_k)&\text{else.} 
        \end{cases}
    \]
    We now aim to find $\ell$ to verify~\eqref{eq:QuasiL2ProjectionEll}. For this, note that shape regularity bounds the number of simplices sharing a common vertex. By construction, 
    $(I_j v -v)_{|\tilde T} = 0$
    for all $\tilde T\in \tilde{\cT}$ that do not share a vertex with any $T\in \hat\cT_j\cap \tilde\cT$. 
    Conversely, consider a $\tilde T\in \tilde{\cT}$ that does share a vertex with some $T\in \hat\cT_j\cap \tilde\cT$. 
    Let $T = T_1, \ldots, T_{\tilde \ell} = \tilde T$ be a sequence of simplices sharing the common vertex, and  the consecutive elements $T_{j}, T_{j+1}$ share a~$d-1$~dimensional face. 
    Then $\abs{\level(T_j) - \level(T_{j+1})}\leq 1$ and thus $\abs{\level(T) - \level(\tilde T)}\leq \tilde \ell -1$. 
    Note that $\tilde \ell$ is bounded by the maximum number of simplices sharing a vertex. This only depends on shape regularity and is thus uniformly bounded by some $\ell$. 
    The property~\eqref{eq:QuasiL2ProjectionEll} follows.

    To establish~\eqref{eq:QuasiL2Projection}, we split $\tilde\cT$ into the three disjoint subsets
    \[       
        \tilde\cT_j = \tilde\cT\cap \hat\cT_j,\quad
        \tilde\cT^{(0)} = \{T\in\tilde\cT\setminus \tilde\cT_j\colon (I_j v-v)_{|T} = 0\},\quad
        \tilde\cT^{*} = \tilde\cT\setminus(\tilde\cT_j \cup \tilde\cT^{(0)}).
    \]
    Let $\tilde T\in \tilde\cT^{*}$. 
    Then there is $T\in \hat\cT_j$ with $T\subset\tilde T$ and that shares the vertices of $\tilde T$ on which $I_jv-v$ does not vanish.
    Let $\Psi\colon \tilde T\to T$ be an affine linear transformation. Then using \Cref{lem:L2normTriangle} we get
    \[
        \norm{Q_jv-v}^2_{L_2(T)}
        =
        \frac{\vol(T)}{\vol(\tilde T)}\norm{(Q_jv-v)\circ \Psi}^2_{L_2(\tilde T)}
        \geq
        \frac{c_d\vol(T)}{C_d\vol(\tilde T)}\norm{I_jv-v}^2_{L_2(\tilde T)} 
        \]
        since the values $I_jv-v$ on the vertices either agree with those of $(Q_jv-v)\circ \Psi$ or are zero. Finally, for any $v\in V(\mathcal T)$ we have the estimate
    \[
    \begin{aligned}
        \norm{ I_j v -v}_{L_2(D)}^2 
        &= \sum_{\tilde T\in \tilde\cT^*}
        \norm{ I_j v -v}_{L_2(\tilde T)}^2
        + 
        \sum_{\tilde T\in \tilde\cT_j}
        \norm{ Q_j v -v}_{L_2(T)}^2 
        \\
        &\leq
        \frac{C_d\vol(\tilde T)}{c_d\vol(T)}
        \sum_{ T\in \hat\cT_j\setminus\tilde\cT_j}
        \norm{ Q_j v -v}_{L_2( T)}^2
        + 
        \sum_{T\in \tilde\cT_j}
        \norm{ Q_j v -v}_{L_2(T)}^2 
        \\
        &\leq c_\ell\norm{ Q_j v -v}_{L_2(D)}^2,
    \end{aligned}
    \]
    where $c_\ell =  \frac{C_d\vol(\tilde T)}{c_d\vol(T)}$ depends only on $\ell$ and shape regularity.
\end{proof}

We also require that the change from basis of uniform basis functions to hierarchical basis functions behaves well. For this, recall that the nodes in $\hat\cT_j$ are enumerated by $\cN_j\supset \cN_{j-1}$.
\begin{lemma}\label{lem:basis_change}
    Let $u_j = \sum_{k\in\cN_j}z_k \psi_{j,k} = \sum_{k\in\cN_{j-1}}t_k \psi_{j-1,k} +
    \sum_{k\in \cN_j\setminus\cN_{j-1}}t_k \psi_{j,k} \in V(\hat\cT_j)$. Then there are constants $c,C>0$ independent of $j$ and $u_j$ such that 
    \[
      c\sum_{k\in \cN_j}\abs{z_k}^2 
      \leq
      \sum_{k\in \cN_j}\abs{t_k}^2
      \leq
      C\sum_{k\in \cN_j}\abs{z_k}^2.
    \] 
\end{lemma}
\begin{proof}
  For $k\in\cN_{j-1}$ we denote by $k_{(1)}, \ldots, k_{(n_k)}\in\cN_j\setminus\cN_{j-1} $ the indices of nodes such that $(x_k, x_{k_{(i)}})$ are edges in $\hat\cT_j$ for $i = 1, \ldots, n_k$. 
  On the other hand, for $k\in\cN_j\setminus\cN_{j-1} $ we denote by $k^{(1)}, k^{(2)}\in\cN_{j-1}$ the indices of nodes in $\hat\cT_{j-1}$ such that $x_k = \frac12(x_{k^{(1)}}+x_{k^{(2)}})$.

As a direct consequence of the triangulations being generated by bisection, we get
the refinement relation
\[
\psi_{j-1,k}
=
r_{j,k,k}
\psi_{j,k}
+
r_{j,k,k_{(1)}}
\psi_{j,k_{(1)}}
  +\ldots +
  r_{j,k,k_{(n_k)}}
  \psi_{j,k_{(n_k)}}
\]
with
\[
  r_{j,k,k}= 
  \frac{\psi_{j-1,k}(x_k)}{\psi_{j,k}(x_k)}
  \quad\text{and}\quad
  r_{j,k,k_{(i)}}=
  \frac{\psi_{j-1,k}(x_{k_{(i)}})}{\psi_{j,k_{(i)}}(x_{k_{(i)}})}.
\]
Note that there is a uniform bound $n \geq n_k$ for all $k$ by uniform shape regularity, and that the ratios satisfy
\[
  0<\tilde c\leq r_{j,k,k_{(i)}} \leq \tilde C
\]
with uniform constants $c, C$
due to the triangulation $\hat\cT_j$ being uniform.
On the other hand, for each index $k\in\{ N_{j-1}+1,\ldots, N_j\}$ in at most two such refinement relations, namely for ${k^{(1)}},{k^{(2)}}\in\{1, \ldots, N_{j-1}\}$.
We thus get
\[
\begin{aligned}
    \sum_{k\in\cN_{j-1}}&t_k \bigl(
    r_{j,k,k}\psi_{j,k}
    +
    r_{j,k,k_{(1)}}
    \psi_{j,k_{(1)}}
    +\ldots +
    r_{j,k,k_{(n_k)}}
    \psi_{j,k_{(n_k)}}
    \bigr)
    +
    \sum_{k\in\cN_j\setminus\cN_{j-1}}t_k \psi_{j,k}\\
&=
\sum_{k\in\cN_{j-1}}b_k 
r_{j,k,k}\psi_{j,k}
+
\sum_{k\in\cN_j\setminus\cN_{j-1}}(b_k +
t_{k^{(1)}}r_{j,k^{(1)},k}
+ t_{k^{(2)}} r_{j,k^{(2)},k}
) \psi_{j,k}
=\sum_{k\in\cN_{j}} z_k \psi_{j,k}
.
\end{aligned}
\]
It follows by Schur's lemma, that 
\[
\sum_{k\in\cN_j}\abs{z_k}^2
\leq
\sum_{k\in\cN_j}\abs{t_k}^2
\max
(
   n \tilde C , 1
)(1+2\tilde C)
\]
and similarly, using 
\[
\psi_{j,k}
=
\frac{1}{
r_{j,k,k}}
\psi_{j-1,k}
-
\frac{
r_{j,k,k_{(1)}}}{
  r_{j,k,k}}
\psi_{j,k_{(1)}}
  -\ldots -\frac{
  r_{j,k,k_{(n_k)}}}{
    r_{j,k,k}}
  \psi_{j,k_{(n_k)}}
\]
we get
\[
\sum_{k\in\cN_j}\abs{t_k}^2
\leq
\sum_{k\in\cN_j}\abs{z_k}^2
\frac{1}{\tilde{c}^2}
\max
(
   n \tilde C , 1
)(1+2\tilde C). \qedhere
\]
\end{proof}

Now we have the necessary tools to prove \Cref{thm: conforming is stable}.

\begin{proof}[Proof of \Cref{thm: conforming is stable}]
  Let $S$ be conforming with corresponding conforming simplicial mesh $\cT$.
  As a consequence of~\eqref{eq: H1 equiv multilevel} and \eqref{eq:QuasiL2Projection}, for any $v\in V(S)$,
  \begin{equation}\label{eq:projectionH1Equivalence}
    \sum_{j\in\N_0} 2^{2j} \norm{(I_j - I_{j-1}) v}_{L_2(D)}^2  \lesssim \norm{v}_{H_0^1(D)}^2 ,
  \end{equation}
 where $I_{-1} v  = 0 $.  
  We construct the coefficients $z_\lambda$ of a stable representation 
  \begin{equation}\label{eq:stabledecompproof}
  v = \sum_{\lambda\in S} z_\lambda \psi_{\lambda} \in V(S)
  \end{equation}
  in multiple steps. For each $j \in \N_0$, let $(\hat z_{j,k})_{k \in \mathcal{N}_j}$ be such that 
  \[
  \sum_{k\in\cN_j}\hat z_{j,k} \psi_{j,k} = (I_j-I_{j-1}) v \in V(S)\cap V(\hat\cT_j).
  \]
By \eqref{eq:levelRieszProperty}, we have $\sum_{k\in \cN_j}\abs{\hat z_{j,k}}^2 \leq 2^{2j} \norm{(I_j-I_{j-1}) v}_{L_2(D)}^2$.
  However, this does not yet yield \eqref{eq:stabledecompproof}, since not all nonzero coefficients $\hat z_{j,k}$ satisfy $(j,k)\in S$. 
Apply $\ell$ changes of basis of the form of \Cref{lem:basis_change} yields the coefficients
  \[
  \sum_{k\in\cN_{j-\ell}}\tilde z_{j, j - \ell,k} \psi_{j-\ell,k} +
  \sum_{i = 0}^{\ell-1} \sum_{k\in\cN_{j-i}\setminus\cN_{j-i-1}}\tilde z_{j,j-i,k} \psi_{j-i,k} = (I_j-I_{j-1}) v \in V(S)\cap V(\hat\cT_j).
  \]
  We utilize the condition~\eqref{eq:QuasiL2ProjectionEll} to deduce
  \[
  (I_j -I_{j-1})v = 0 \quad\text{on $T\in\cT$ with $\level T \leq j-1-\ell$}.
  \] 
  Thus, the coarser basis functions $\psi_{i,k}$ with $i\leq j - 1 -\ell$ do not enter into the expansion of $(I_j-I_{j-1}) v$.
  This in turn implies $\tilde z_{j,j-i,k} = 0$ for all $(j-i,k)\notin S$.
  Finally, we obtain the sought coefficients in \eqref{eq:stabledecompproof} by setting
  \[
  z_{j,k} = \sum_{i = 0 }^\ell \tilde z_{j+i,j, k},\quad\text{so that}\quad
  \sum_{\lambda\in\Lambda} z_\lambda \psi_\lambda = \sum_{j}(I_j - I_{j-1})v  =v.
  \]
  Using the Cauchy-Schwarz inequality and \Cref{lem:basis_change}, we have the estimate
  \[
  \sum_{\lambda\in S} \abs{z_\lambda}^2 
  \leq 
  (\ell + 1) 
  \sum_{(j,k\in S)} \sum_{i = 0}^\ell \abs{\tilde z_{j + i , j , k}}^2
  \leq 
  (\ell + 1) C^\ell
  \sum_{j\in\N_0} \sum_{k\in\cN_j} \abs{\hat z_{j , k}}^2
  \]
  which together with~\eqref{eq:projectionH1Equivalence} implies the desired stability.
\end{proof}

\subsection{Proof of optimality theorem}\label{sec:proofOptimalityResult}

In the refinement process described in~\cite[\sect{3.2}]{bachmayr2024convergent} we utilized a tree structure on the index set. It is now possible to introduce a tree structure on $\Theta$ such that any conforming set $\Theta$ satisfies the tree structure, that is $\lambda \in \Lambda$ implies that all ancestors $\lambda'$ of $\lambda$ are in $\Lambda$ as well. 
We choose the ancestor relation defined as follows: $\lambda'=(j-1,k)$ is the parent of $\lambda$ if either $\lambda= (j,k)$ or $\lambda = (j,k')$ where $x_{k'}$ is the midpoint of an edge with endpoint $x_k$.
This implies the conditions stated in~\cite[\sect{3.2}]{bachmayr2024convergent}.

To provide estimates of the number of refinements, we also require $\omega_0 < \omega_1$ and
\begin{equation*}
  \omega_1(1-\zeta) +\zeta < (1-2\zeta) {\frac{c_\mathrm{stable}\sqrt{c_B}}{C_\Psi \sqrt{C_B}}}
\end{equation*}
previously stated as~\eqref{eq:omega1Condition}.
The following lemma is in virtue similar to~\cite[\lem{5.3}]{BV} with the additional difficulty that we have to combine conforming sets and sets with tree structure.

\begin{lemma}\label{lem:LambdaEstimate}
  Let $\Lambda^{k+1}$ be obtained by \Cref{alg:AdaptiveMethod}.
  In addition to \eqref{eq:zetaCondition}, \eqref{eq:omega0Condition}, \eqref{eq:gammaCondition}, let $\omega_1$ satisfy \eqref{eq:omega1Condition}.
  Then 
  \[
  \# (\Lambda^{k+1}\setminus \Lambda(\T^k)) 
  \lesssim  
  \min \Bigl\{
  \#\bar\Lambda\colon \text{$\bar\Lambda$ is conforming and  } \min_{v\in \cV(\bar\Lambda)}\norm{u-v}_\cV \leq \beta \norm{u-u^k}_B
  \Bigr\}
  \]
  where the hidden constant only dependents on $\omega_0$ and $\omega_1$ and
  with $\beta>0$ such that 
  \[
  (1-C_B\beta^2)^{\frac{1}{2}} =  \biggl(\frac{\omega_1(1-\zeta) +\zeta}{1-2\zeta}\biggr) \frac{C_\Psi \sqrt{C_B}}{c_\mathrm{stable}\sqrt{c_B}}.
  \]
\end{lemma}
\begin{proof}
  We set $\hat\omega = \frac{\omega_1(1-\zeta)+\zeta}{1-2\zeta}$. Then \eqref{eq:omega1Condition} ensures the relation between $\hat\omega$ and $\beta$ required in \Cref{lem:AbstractOptimality} and $\omega_1 = (\hat\omega- \frac{(1+\hat\omega)\zeta}{1-\zeta})$.
  Let $\hat\Lambda\supset\Lambda(\T^k)$ be the smallest conforming set satisfying 
  \[
  \norm{\Psi(Bu-f)|_{\hat\Lambda}}_{\ell_2} \geq \hat\omega \norm{\Psi(Bu-f)}_{\ell_2}.
  \]
  Then 
  \begin{align*}
      \norm{\hat\br^k|_{\hat\Lambda}}_{\ell_2} 
      &\geq  
      \norm{\Psi(Bu-f)|_{\hat\Lambda}}_{\ell_2}  - \norm{\hat\br^k - \Psi(Bu-f)}_{\ell_2}\\
      &\geq \hat\omega \norm{\Psi(Bu-f)}_{\ell_2}  - \norm{\hat\br^k - \Psi(Bu-f)}_{\ell_2}\\
      &\geq \hat\omega \norm{\br^k}_{\ell_2}  - (1+\hat\omega)\norm{\hat\br^k - \Psi(Bu-f)}_{\ell_2}\\
      &\geq (\hat\omega- \frac{(1+\hat\omega)\zeta}{1-\zeta})\norm{\hat\br^k}_{\ell_2}\\
      &= \omega_1  \norm{\hat\br^k}_{\ell_2}
  \end{align*}
and therefore 
\[
  \# (\hat\Lambda\setminus \Lambda(\T^k)) \leq  \min \Bigl\{
  \#\bar\Lambda\colon \text{$\bar\Lambda$ is conforming and  } \min_{v\in V(\bar\Lambda)}\norm{u-v}_V \leq \beta \norm{u-u^k}_B
  \Bigr\}
  \]
by \Cref{lem:AbstractOptimality}. Now tree estimation (cf. \cite[\corol{4.20}]{BV}) ensures 
\begin{equation}\label{eq:treesetsestimate}
\# (\Lambda^{k+1}\setminus \Lambda(\T^k)) \lesssim  \# (\Lambda_T\setminus \Lambda(\T^k))
\end{equation}
for all $\Lambda_T\supset \Lambda(\T^k)$ with tree structure
 satisfying
\[
  \norm{\hat\br^k|_{\Lambda_T}}_{\ell_2} \geq \omega_1 \norm{\hat\br^k}_{\ell_2}.
\]
Here, the hidden constant in~\eqref{eq:treesetsestimate} only depends on $\omega_0$ and $\omega_1$.
Since the smallest $\Lambda_T$ is smaller than $\hat\Lambda$, we get
\[
  \# (\Lambda^{k+1}\setminus \Lambda(\T^k)) \lesssim  \min \Bigl\{
      \#\bar\Lambda\colon \text{$\bar\Lambda$ is conforming and  } \min_{v\in V(\bar\Lambda)}\norm{u-v}_V \leq \beta \norm{u-u^k}_B
      \Bigr\}
\]
as asserted.
\end{proof}

The last necessary ingredient is the relation of the sizes of the sets $\Lambda^k$ and the size of the meshes $\T^k$.

\begin{lemma}\label{lem:amortizedCosts}
  Let $\T^k = \msx{\Lambda^k}$ and $\Lambda^k$ be generated by \Cref{alg:AdaptiveMethod}. Then 
  \[
       N(\T^k) -  N(\T^0) \lesssim \sum_{j = 1}^k \#(\Lambda^{j}\setminus\Lambda(\T^{j-1})).
  \]
\end{lemma}

 \begin{proof}
 By the tree structure of $\Lambda^{j}$, each element of $\Lambda^{j}\setminus\Lambda(\T^{j-1})$ corresponds to applying to a certain mesh element at most~$d$ successive bisections with conforming closures of the respective meshes. Following similar lines as in \cite[\propo{3.10}]{bachmayr2024convergent}, the result thus follows with the amortized complexity bound for bisection given in \cite[\theo{6.1}]{Stevenson:08}.
 \end{proof}

\begin{proof}[Proof of \Cref{thm: quasioptimal approximation}]
By \Cref{lem:amortizedCosts} and \Cref{lem:LambdaEstimate}, we get 
\[
\begin{aligned}
   N(\T^k) &-  N(\T^0) 
  \lesssim \sum_{j = 1}^k \#(\Lambda^{j}\setminus\Lambda(\T^{j-1}))\\
  &\lesssim \sum_{j = 1}^k \min \Bigl\{
    \#\bar\Lambda\colon \text{$\bar\Lambda$ is conforming and  } \min_{v\in \cV(\bar\Lambda)}\norm{u-v}_\cV \leq \beta \norm{u-u^{j-1}}_B
    \Bigr\}.
\end{aligned}
\]
We next note that for any conforming simplicial mesh,
\[
    \#\Lambda(\mathcal T)  +1\eqsim \dim V(\mathcal T) +1 \eqsim \#\mathcal T 
\]
with constants depending on shape regularity. Thus, 
\begin{multline*}
  \min \Bigl\{  
    \#\bar\Lambda\colon \text{$\bar\Lambda$ is conforming and  } \min_{v\in \cV(\bar\Lambda)}\norm{u-v}_\cV \leq \beta \norm{u-u^{j-1}}_B
  \Bigr\}\\
  \eqsim
  \min \Bigl\{  
    \# N(\T)\colon \min_{v\in \mathcal \cV(\T)}\norm{u-v}_\cV \leq \beta \norm{u-u^{j-1}}_B
  \Bigr\}\\
  \leq \norm{u}_{\mathcal A^s}^{\frac1s} (\beta \norm{u - u^{j-1}}_B)^{-\frac1s}.
\end{multline*}
Finally, using the error reduction~\eqref{eq:ErrorReduction} of \Cref{thm:ErrorReduction} together with previous estimates,  we get 
\[
  N(\T^k) -  N(\T^0) \lesssim 
  \norm{u}_{\mathcal A}^{\frac1s}
  \norm{u - u^{k}}_B^{-\frac1s} \beta^{-\frac1s} 
  \sum_{j = 1}^k (\delta^\frac1s)^{1+k - j}
  \leq 
  \norm{u}_{\mathcal A^s}^{\frac1s}
  \norm{u - u^{k}}_\cV^{-\frac1s} \beta^{-\frac1s}c_B^{-\frac1{2s}} 
  \frac{1}{1- \delta^{\frac1s}},
\]
which is~\eqref{eq:quasi_optimal_approximation} since $0<c_B\beta^2 \leq \frac{c_B}{C_B}\leq 1$. This concludes the proof of (i).

The method \textsc{TreeApprox} scales linear in the size of $\Lambda^k$, and the methods \textsc{mesh} and \textsc{GalerkinSolve} scale linear in the size $N(\T^k)$. They are thus dominated by the costs of $\textsc{ResEstimate}$.

The costs of $\textsc{ResEstimate}$ are estimated in \Cref{prop:ResEstimate}. 
By invoking the first part of the theorem, we can estimate the approximation norm $\norm{u^k}_{\cA^s}$ in terms of $\norm{u}_{\cA^s}$ via
\begin{align*}
  \norm{u^k}_{\cA^s}
  &=
  \max_{N< N(\T^k)}
  N^s 
  \min_{N(\mathbb T)\leq N}
  \min_{u_N\in \mathcal V(\mathbb T)}\norm{u^k-u_N}_\mathcal V
  \\
  &\leq
  \max_{N< N(\T^k)}
  N^s 
  \min_{N(\mathbb T)\leq N}
  \min_{u_N\in \mathcal V(\mathbb T)}(\norm{u-u^k}_\cV+\norm{u-u_N}_\mathcal V)
  \\
  &\leq
  N(\T^k)^s\norm{u-u^k}_\cV + \norm{u}_{\cA^s}.
\end{align*}
We also get
\[
  \bignorm{\bigl(\norm{[u^k]_\nu}_V \bigr)_{\nu\in\cF}}_{\cA^s}
  \lesssim
  \norm{u^k}_{\cA^s}
  \leq
  \norm{u}_{\cA^s} + N(\T^k)^{s}\norm{u-u^k}_{\cV}
  \lesssim
  \norm{u}_{\cA^s}.
\]
By~\eqref{eq:ropest}, the number of operations is thus bounded by a fixed multiple of 
\[
\sum_{j = 1}^k
q(1+\abs{\log\eta_j} +\log{\norm{u}_{\cA^s}})
      \bigl(
        \#F^j\log\#F^j +
      \#\cT(f) + \eta_j^{-p} \norm{u}_{\cA^s}^{p} N(\T^j) + \eta_j^{-\frac1s} \norm{u}_{\cA^s}^{\frac1s}\bigr)
\]
for some polynomial $q$. Again, by the first part of the theorem, we can estimate 
$N(\T^k)\lesssim \norm{u-u^k}_\cV^{-\frac1s}\norm{u}_{\cA^s}^{\frac1s}$ and by error reduction 
$\delta^{k-j}\norm{u-u^j}_\cV \lesssim\norm{u-u^k}_\cV $.
Also, by construction $\eta_k\eqsim \norm{u-u^k}_{\cV}$ and $N(\T^0)\lesssim 1$ and $\#\cT(f)\lesssim 1$. 
The term $\#F^j\log\#F^j$ is dominated by $N(\T^j)\log N(\T^j)\lesssim \norm{u-u^j}_\cV^{-p-\frac1s}\norm{u}_{\cA^s}^{p+\frac1s}$.
Thus, the number of operations is bounded by a fixed multiple of
\begin{multline*}
  q\bigl(1+\bigabs{\log\norm{u-u^k}_\cV} +\log{\norm{u}_{\cA^s}}\bigr)
\sum_{j = 1}^k
      \bigl(
      1 + \delta^{(p+\frac1s)(j-k)}\norm{u-u^k}_\cV^{-p-\frac1s} \norm{u}_{\cA^s}^{p+\frac1s}\bigr)\\
      \leq
      q\bigl(1+\bigabs{\log\norm{u-u^k}_\cV} +\log{\norm{u}_{\cA^s}}\bigr)
      \Bigl(k + (1-\delta^{p+\frac1s})^{-1}\norm{u-u^k}_\cV^{-p-\frac1s} \norm{u}_{\cA^s}^{p+\frac1s}\Bigr).
\end{multline*}
Finally, by error reduction $k\lesssim \frac{1}{\log\delta} (\log\norm{u-u^k}_\cV -\log\norm{u}_\cV)$. This concludes the proof of~(iii). For~(ii) we use the same technique but the computational complexity of the method \textsc{ResEstimate} in the affine case stated in~\cite[\propo{3.7}]{bachmayr2024convergent}.
\end{proof}

\section{Conclusion and Outlook}\label{sec:Conclusion}

We have extended the adaptive stochastic Galerkin finite element method developed in \cite{bachmayr2024convergent} to handle nonlinear parametrizations of coefficients, more specifically, allowing for series expansions composed with a general class of smooth nonlinearities. In particular, this applies to the exponential function, covering log-affine parametrizations.
This treatment of nonlinear parametrizations can be accommodated within the basic structure of the method from \cite{bachmayr2024convergent}, which is based on a combination of adaptive operator compression in the stochastic variables with spatial error estimation using finite element frames. The modification for nonlinear coefficient parametrizations requires modified compression strategies, such that the results on uniform error reduction by each iteration of the adaptive method shown in  \cite{bachmayr2024convergent} remain valid.

Both for affine and for nonlinear parametrizations, we have shown that the adaptive scheme generates quasi-optimal approximations at near-optimal computational costs. To achieve optimal computational costs, as in \cite{BV,bachmayr2024convergent} we crucially rely on coefficient parametrizations with multilevel structure, which leads to suitable approximate sparsity of solutions and operators.

There are several natural directions for further work. First, the nonlinear coefficient parametrizations are a natural fit for applications to random domain problems with transformation to a reference domain as, e.g., in \cite{MR3563282,MR4799248}. Second, the generalization of the adaptive scheme to vector valued elliptic problems and to parabolic problems will be of interest. Finally, we also plan to investigate adaptive methods based on similar concepts for goal-oriented discretization refinement.

\appendix
\section{Proof of auxiliary lemmas}\label{sec:lemmas}

\begin{proof}[Proof of \Cref{lem: Stoch Schur Lemma}]
The result follows from the estimates 
\begin{align*}
\norm{Bv}_\cV
& \leq \Bigl(\sum_{\nu\in \mathcal F} \Bignorm{\sum_{\nu'\in \mathcal F} [a]_{\nu\nu'}\nabla [v]_{\nu'}}^2_{L_2(D)}\Bigr)^{\frac12}\\
&\leq \biggl(\sum_{\nu\in \mathcal F} \Bignorm{\Bigl(\sum_{\nu'\in \mathcal F} \abs{[a]_{\nu\nu'}}\abs{\nabla [v]_{\nu'}}^2\Bigr)\Bigl(\sum_{\nu'} \abs{[a]_{\nu\nu'}}\Bigr)}_{L_1(D)}\biggr)^{\frac12}\\
&\leq \Bignorm{\sum_{\nu\in \mathcal F}\Bigl(\sum_{\nu'\in \mathcal F} \abs{[a]_{\nu\nu'}}\abs{\nabla [v]_{\nu'}}^2\Bigr)}^\frac12_{L_1(D)} \max_{\nu\in\mathcal F}\Bignorm{\sum_{\nu'\in \mathcal F} \abs{[a]_{\nu\nu'}}}_{L_\infty(D)}^{\frac12}\\
&\leq 
\max_{\nu'\in\mathcal F}\Bignorm{\sum_{\nu\in \mathcal F} \abs{[a]_{\nu\nu'}}}_{L_\infty(D)}^{\frac12}
\Bignorm{\sum_{\nu\in \mathcal F}\abs{\nabla [v]_{\nu }}^2}^\frac12_{L_1(D)} \max_{\nu\in\mathcal F}\Bignorm{\sum_{\nu'\in \mathcal F} \abs{[a]_{\nu\nu'}}}_{L_\infty(D)}^{\frac12},
\end{align*}
since by symmetry $[a]_{\nu\nu'}=[a]_{\nu'\nu}$, we arrive at
\[   \norm{Bv}_\cV \leq \max_{\nu\in\mathcal F}\Bignorm{\sum_{\nu'\in \mathcal F} \abs{[a]_{\nu\nu'}}}_{L_\infty(D)} \norm{v}_\cV . \qedhere \]
\end{proof}

  \begin{proof}[Proof of \Cref{lem: Stechkin}]
  First, we introduce the functions
  \[
    a(x) = a_i \quad\text{if $x\in (i-1,i]$,}\quad
    b(x) = b_i \quad\text{if $x\in (i-1,i]$,} 
    \quad\text{and}\quad
    B(x) = \int_0^x b(y) \sdd y.
  \]
  Note that $a(x)$ is monotonically decreasing and $B(x)$ is strictly monotonically increasing. Thus $B\colon [0,\infty) \to [0,\infty)$ is bijective. Then 
  \[
  \Bigl(\sum_{i = k+1}^\infty a_i b_i\Bigr) \Bigl(\sum_{i = 1}^k b_i\Bigr)^{-1}
  = \int_k^\infty a(x) b(x)\frac1{B(k)}\sdd x
  = \sum_{n=1}^{\infty} \int_{n B(k)}^{(n+1)B(k)}\frac{a(B^{-1}(y))}{B(k)} \sdd y.
  \] 
  Using $0<p<1$ and monotonicity of $a$, we obtain
  \begin{align*}
  \Bigl(\sum_{i = k+1}^\infty a_i b_i\Bigr) \Bigl(\sum_{i = 1}^k b_i\Bigr)^{-1}
  &\leq
  \biggl(\sum_{n=1}^{\infty} \biggl(\int_{n B(k)}^{(n+1)B(k)}\frac{a(B^{-1}(y))}{B(k)} \sdd y\biggr)^p\biggr)^\frac1p
 \leq 
  \biggl(\sum_{n = 1}^\infty a(B^{-1}(nB(k)))^p\biggr)^\frac1p\\
  &\leq \biggl(\int_0^\infty \frac{a(B^{-1}(y))^p}{B(k)}\sdd y\biggr)^\frac1p
  = \biggl(\int_0^\infty \frac{a(x)^pb(x)}{B(k)}\sdd y\biggr)^\frac1p\\
  &\leq \Bigl(\sum_{i = 1}^\infty a_i^p b_i\Bigr)^\frac1p \Bigl(\sum_{i = 1}^k b_i\Bigr)^{-\frac1p}. \qedhere
  \end{align*}
  \end{proof}

  \begin{proof}[Proof of \Cref{lem:SummabilityOfRefSequence2}]
    Similarly to the proof of \Cref{lem: summability of reference sequence}, we obtain 
    \begin{align*}
    \sum_{\mathbf k\in \mathcal F(\N_0)}
    a_\mathbf k^p b_\mathbf k
    &= a_\mathbf 0^p b_\mathbf 0 + 
    \sum_{L = 0}^\infty
    \sum_{\mathbf k\in \mathcal F(\N_0,L)}
    a_{\mathbf k + e_L}^p b_{\mathbf k + e_L}\\
    &= 1 + 
    \sum_{L = 0}^\infty
   \rho^{-p} 2^{(d_b-(\alpha+d_a)p)L}
    \sum_{\mathbf k \in \mathcal F(\N_0,L)}
    2^{\# \{\ell\colon k_\ell > 0\}}
    \prod_{\ell=0}^L (\rho^p 2^{\alpha\ell p })^{-k_\ell}
    \binom{k_\ell +t}{t}
    \\
    &= 1 + 
    \sum_{L = 0}^\infty
    \rho^{-p} 2^{(d_b-(\alpha+d_a)p)L}
    \bigg(
    2
    \sum_{k =0}^\infty
    (\rho^p 2^{\alpha \ell p })^{-k}
    \binom{k +t}{t}
    -1
    \bigg).
    \end{align*}
    Since $\sum_{k = 0}^\infty \binom{k+t}{t} x^k = \frac{1}{(1-x)^{t+1}}$ for $\abs{x}<1$, we have 
  \[
    \sum_{\mathbf k\in \mathcal F(\N_0)}
    a_\mathbf k^p b_\mathbf k
    =
    1 + 
    \rho^{-p}
    \sum_{L = 0}^\infty
    2^{(d_b-(\alpha+d_a)p)L}
    \prod_{\ell=0}^L 
    \frac{2 - (1-\rho^{-p} 2^{-\alpha \ell p })^{t+1}}{(1- \rho^{-p} 2^{-\alpha \ell p })^{t+1}},
  \]
  and the product is bounded 
    \begin{align*}
      \prod_{\ell=0}^L 
      \frac{2 - (1-\rho^{-p} 2^{-\alpha \ell p })^{t+1}}{(1- \rho^{-p} 2^{-\alpha \ell p })^{t+1}}
      &\leq
      \exp\sum_{\ell=0}^\infty \bigl(
      \ln
      (2 - (1-\rho^{-p} 2^{-\alpha \ell p })^{t+1})-(t+1)\ln(1- \rho^{-p} 2^{-\alpha \ell p })\bigr)\\
      &\leq
      \exp\biggl(\sum_{\ell=0}^\infty
      \rho^{-p}(t+1)\biggl( 1+ \frac{1}{1- \rho^{-p}} \biggr)2^{-\alpha \ell p }\biggr)\\
      &=
      \exp\biggl(
      \frac{\rho^{-p} (t+1)(2 -\rho^{-p})}{(1- \rho^{-p})(1-2^{-\alpha  p })}\biggr)
    \end{align*}
  where we used $\ln
  (2 - (1-x)^{t+1})\leq (t+1)x$ and $\ln(1-\rho^{-p}x)\leq \frac{x}{1-\rho^{-p}}$ for $x\in(0,1)$. The result follows.
  \end{proof}

\label{sec:conclusion}

\bibliographystyle{amsplain}
\bibliography{literature}

\end{document}